\title{Doubly Reflected BSDEs in the predictable setting }
       
\author{ Ihsan Arharas\thanks{Cadi Ayyad University, Department of mathematics, Faculty of Sciences Semlalia,
B.P. 2390, Marrakesh, Morocco.
E-mail: ihsan.arharas@edu.uca.ac.ma}
 \and
Siham Bouhadou \thanks{
                  Cadi Ayyad University, Department of mathematics, Faculty of Sciences Semlalia,
B.P. 2390, Marrakesh, Morocco.
E-mail: sihambouhadou@gmail.com}
      \and 
        Youssef Ouknine              
                      \thanks{Mohammed VI Polytechnic University, Africa Business School, Lot 660, Hay Moulay Rachid, Ben Guerir 43150, Morocco, and 
Cadi Ayyad University,
Department of mathematics, Faculty of Sciences Semlalia,  
B.P. 2390, Marrakesh, Morocco.
E-mail: youssef.ouknine@um6p.ma, ouknine@uca.ac.ma}
        }

\documentclass[12pt,a4paper,openright]{article}
\usepackage[left = 3.5 cm, right = 3.5 cm, top = 3.5cm, bottom = 3cm]{geometry}
\usepackage{graphicx}
\usepackage[utf8]{inputenc}
\usepackage[T1]{fontenc}
\usepackage{mathptmx}
\usepackage{ntheorem}
\usepackage[colorlinks=true]{hyperref}
\usepackage{color}
\hypersetup{colorlinks = true, linkcolor = red, citecolor= green }
\usepackage{amsmath}
\usepackage{amsmath,amssymb} 
\usepackage{mathtools} 
\numberwithin{equation}{section} 
\newcommand\numberthis{\addtocounter{equation}{1}\tag{\theequation}}

\newcommand{\vertiii}[1]{{\left\vert\kern-0.25ex\left\vert\kern-0.25ex\left\vert #1 
    \right\vert\kern-0.25ex\right\vert\kern-0.25ex\right\vert}}
\DeclareMathOperator*{\esssup}{ess\,sup}
\newtheorem{theorem}{Theorem}[section]
\newtheorem{lemma}{Lemma}[section]
\newtheorem{acknowledgements}{Acknowledgement}[section]
\newtheorem{notation}{Notation}[section]
\newtheorem{proposition}{Proposition}[section]

\newtheorem{remark}{Remark}[section]
\newtheorem{definition}{Definition}[section]
\newenvironment{proofa}{{\sc Proof of Lemma \ref{lim}.}}{}
\newenvironment{proof}{{\sc Proof.}}{}
\newenvironment{AMS}{}{}
\newenvironment{keywords}{}{}

\begin{document}
\newpage
\maketitle
\begin{abstract}
In this paper, we introduce a specific kind of doubly reflected Backward Stochastic Differential Equations (in short DRBSDEs), defined on probability spaces equipped with general filtration that is essentially non quasi-left continuous, where the barriers are assumed to be predictable processes. We call these equations \textit{predictable DRBSDEs}. Under a general type of \textit{Mokobodzki's condition}, we show the existence of the solution (in consideration of the driver's nature) through a Picard iteration method and a Banach fixed point theorem. By using an appropriate generalization of Itô's formula due to Gal'chouk \cite{Gal'chouk} and Lenglart \cite{Lenglart}, we provide a suitable a priori estimates which immediately implies the uniqueness of the solution.
\end{abstract}

\begin{AMS}
\small \textbf{AMS MSC:} \href{https://mathscinet.ams.org/msc/msc2010.html?s=60H20}{60H20}, \href{https://mathscinet.ams.org/msc/msc2010.html?s=60H30}{60H30}, \href{https://mathscinet.ams.org/msc/msc2010.html?s=65C30}{65C30}.  
\end{AMS}


\section{Introduction}

 \hspace*{0.5 cm} The theory of Backward stochastic differential equations (BSDEs, in short) has been widely studied in the literature due to their affiliation with many problems in different mathematical areas. We mention, among others, partial differential equations, theoretical economies, mathematical finance, stochastic optimal control, game theory, and other optimality problems. These equations can be traced back to Bismut \cite{Bismut} who presented them in a linear form as the equation for the conjugate variable in the stochastic Pontryagin maximum principle. Acfterwards, Pardoux and Peng \cite{Pardoux and Peng} generalized them to the nonlinear case when the noise is driven by a Brownian motion. Precisely, given a filtered probability space $(\Omega, \mathcal{F}, \mathbb{F}=(\mathcal{F})_{t \in[0,T]}, P)$ generated by an $\mathbb{R}^d$- valued Brownian motion $W$, a solution for the BSDE associated with data $(g,\xi)$ and terminal time $T$, consists of a pair of measurable processes $(Y, Z)$ in suitable spaces, mainly satisfies: 
\begin{equation} \label{BSDE}
Y_t= \xi + \int_t^T g(s,Y_s, Z_s) ds - \int_t^T Z_s dW_s;  \,\,\,\, \text{for all} \,\,\,\, t\in[0,T],
\end{equation} 
where the generator $g$ is Lipschitz continuous in $(y,z)$ and the terminal value $\xi$ is square-integrable. \\ 
\hspace*{0.5 cm} Thereafter, a new kind of BSDEs called reflected backward stochastic differential equations (RBSDEs, in short), has been introduced by El Karoui et al. \cite{El Karoui and al.} in the case of a Brownian filtration and a continuous obstacle. In their setting, the first component of the solution is forced to remain greater than or equal to a given process called obstacle or barrier. An additional nondecreasing predictable process appeared in the expression (\ref{BSDE}); the function of this additional process is to push upwards the process $Y$ in order to keep it above the barrier $\xi$. One important use of RBSDEs is their application to the pricing of American options, especially in constrained markets.\\ 
\hspace*{0.5 cm} The theory of RBSDEs has been extended to the case where the obstacle is not necessarily continuous and a larger filtration than the Brownian filtration by several authors, we quote \cite{Essaky(2008)}, \cite{Hamadene-Ouknine}, \cite{Quenez-Sulem (2014)}, \cite{Crepey-Matoussi} and references therein. In all of the mentioned works, the barrier has been assumed to be at least right-continuous. In the last few years, financial quantitative analysts have used more sophisticated mathematical concepts, such as tools from the general theory of stochastic processes, in order to describe the behavior of irregular phenomena. Using some techniques and approache from the optimal stopping theory (cf. \cite{El karoui}, \cite{kobylanski}, \cite{maingueneau}) and some results from the  general theory of processes (cf. \cite{Probabilites et Potentiel2}), Grigorova et al. in \cite{Reflected BSDEs when the obstacle is not right-continuous and optimal stopping} were the first to deal with RBSDEs in a general framework, where the obstacle is not necessarily right-continuous. They have proved existence and uniqueness of the solutions to such RBSDEs in appropriate Banach spaces. Grigorova et al. also found that the $Y$-component of the unique solution characterize the value function of an optimal stopping problem in which the risk of a
financial position $\xi$ is assessed by an $f$-conditional expectation $\varepsilon^f(.)$ (where
$f$ is a Lipschitz driver). From a perspective of mathematical finance, this discovery is significant
for the evaluation of American options with irregular pay-off in an imperfect market model, see e.g., Grigorova et al. \cite{Optimal stopping with}.

More recently, Bouhadou and Ouknine \cite{Bouhadou} considered the theory of BSDEs in the predictable setting. That is, where the filtration is non quasi-left continuous and the terminal value $\xi$ belongs to $L^2(\mathcal{F}_{T^-})$. A significant use of these equations is to generate a new family of "non linear expectations", called predictable $g$-conditional expectations. These operators present a crucial tool to study the main problem of their paper, namely the optimal stopping problem. Moreover, Bouhadou and Ouknine \cite{Bouhadou} introduced the theory of reflected BSDEs in the predictable framework, where the lower obstacle is given by a left-limited predictable process. \\
\hspace*{0.5 cm} In the current paper, we generalize the previous equations to the case of two reflecting barrier processes, that is, to a setting where, the filtration is non quasi-left continuous and the solution has to remain between two left-limited, predictable processes $\xi$ and $\zeta$, with $\xi \leq \zeta$ and $\xi_T=\zeta_T$. We establish the existence and uniqueness of the solution, in appropriate Banach space, to the following doubly RBSDE:  
\begin{multline} \label{intro-eq}
Y_\tau=\xi_T + \int_\tau^T g(s,Y_s,Z_s) ds - \int_\tau^T Z_s dW_s - (M_{T^-}-M_{\tau^-}) +  A_T-A_\tau \\ - (A'_T-A'_\tau) + B_{T^-}-B_{\tau^-}-(B'_{T^-}-B'_{\tau^-}), \,\,\,\, \text{for all} \,\,\, \tau\in \mathcal{T}_0^p.
\end{multline}

We call these equations \textit{predictable DRBSDEs}, the solution is given by $(Y,Z,M,$\\$A,B,A',B')$, where $M$ is a square integrable martingale. The predictable non-decreasing processes $A,A', B$ and $B'$ have the role to keep the solution between the two obstacles $\xi$ and $\zeta$. It is important to underline that in modelling, the predictable context is interesting since it gives more information compared to the classical cases; the work \cite{Bouhadou} provided much of the inspiration and motivation for our work. \\ 
\hspace*{0.5 cm} Carrying on the work \cite{El Karoui and al.}, Cvitanić and Karatzas have introduced the doubly reflected BSDEs (DRBSDEs, in short) in the case of continuous obstacles and a Brownian filtration in  \cite{Cvitanic and Karatzas}, and then extended to the case of a not necessarily continuous obstacles and a larger filtration in \cite{Hamadene-Hassani-Ouknine}, \cite{Crepey-Matoussi}, \cite{Essaky-Harraj-Ouknine}, \cite{Generalized Dynkin Games and Doubly reflected BSDEs with jumps}. The first paper dealing with BSDEs with two reflecting barriers that are not right-continuous, is the paper of Grigorova et al. (2018) \cite{Doubly Reflected BSDEs and E-Dynkin games}.  Precisely, motivated by the problem of pricing of game options (derivative contracts that can be terminated by both counterparties at any time before a maturity date $T$ \cite{Game Options in an Imperfect Market with Default}) in the case of imperfections in the market model, the authors showed that a doubly reflected BSDE with general filtration, where the barriers are assumed to be optional processes, admits a unique solution if and only if the so called \textit{Mokobodzki's condition} holds (there exist two supermartingales such that their difference is between $\xi$ and $\zeta$). Under an additional regularity assumptions on $\xi$ and $\zeta$, Grigorova et al. found that the solution is characterized in terms of the value of a corresponding $\varepsilon^f$-Dynkin game, i.e. a game problem over stopping times with (non-linear)
$f$-expectation, where $f$ is the driver of the doubly reflected BSDE. For the general irregular case, the authors formulate "an extension" of the previous $\varepsilon^f$-Dynkin game problem over a larger set of "stopping systems" than the set of stopping times. They demonstrated that the solution of the general DRBSDE with irregular optional barriers coincides with the common value of this extended game.

Inspired by ideas of \cite{Doubly Reflected BSDEs and E-Dynkin games}, under an extended type of \textit{Mokobodzki's condition}, we will show the existence and uniqueness of the solution to the predictable DRBSDE (\ref{intro-eq}). In the proof of our result, we use a Picard iteration method to show the existence of a solution when the driver $g$ does not depend on the solution, and then, in the general case, we construct a contraction that has a fixed point which is the solution of our predictable DRBSDE (\ref{intro-eq}). From Lenglart's theory (1980) \cite{Lenglart}, we recall that Meyer-$\sigma$-fields are embedded between predictable and optional fields. Along with the case of optional and predictable obstacles, we point out that the results of this paper can be extended to the case of Meyer-measurable processes.\\
\hspace*{0.5 cm} The rest of the paper is organized as follows: Section 2 begins by listing necessary notations and definitions. Next, we give the definition of our predictable DRBSDEs and some properties. In section 3, we investigate the question of existence and uniqueness of the solution in the case where the driver $g$ does not depend on $y$, $z$, that is, when it is given by a process $g(t)$. Therefore, we show that in this particular case the solution can be given in terms of the solution of a coupled system of two predictable RBSDEs. Under \textit{Mokobodzki's condition}, by applying a Picard iteration method we show that this system admits a solution, and hence the predictable DRBSDE admits a solution as well. In Subsection \textcolor{red}{3.2}, we provide a suitable a priori estimates, which implies the uniqueness of the solution. The general case is treated in Section 4 by using the a priori estimates of Section 3 and a fixed point argument.


\section{Notations and definitions}

\hspace*{0.5cm} Throughout the paper, we work with a finite time horizon $T>0$, a probability space $(\Omega,\mathcal{F},P)$ and a right-continuous complete filtration $\mathbb{F}=\lbrace\mathcal{F}_t:t\in[0,T]\rbrace$. Essentially, we assume that the filtration $\mathbb{F}$ is not quasi-left continuous. Let $W$ be a one-dimensional $\mathbb{F}$-Brownian motion.

\begin{notation}
 We denote by $\mathcal{P}$ (resp. $\mathcal{O}$) the predictable (resp. optional) $\sigma$-algebra on $\Omega \times [0,T]$. We recall that a stopping time $\tau$ is called predictable if there exists a sequence $(\tau_n)_{n \in\mathbb{N}}$ of stopping times that are strictly smaller that $\tau$ on $\lbrace \tau>0 \rbrace$ and increase to $\tau$ a.s. Moreover, we denote by $\mathcal{T}^p_0$ the set of all predictable stopping times $\tau$ with values in $[0,T]$. More generally, for a given stopping time $S$ in $\mathcal{T}^p_0$, we denote by $\mathcal{T}^p_S$ (resp. $\mathcal{T}^p_{S^+}$) the set of predictable stopping times $\tau$ in $\mathcal{T}^p_0 $ such that $S \leq \tau$ a.s. (resp. $S<\tau$ a.s. on $\lbrace S<T \rbrace$ and $\tau =T$ a.s. on $\lbrace S=T \rbrace$).\\

The following spaces will be frequently used in the sequel.

\begin{enumerate}
\item[•] $L^2(\mathcal{F}_{T^{-}})$ is the set of random variables which are $\mathcal{F}_{T^{-}}$-measurable and square-integrable. 
\item[•] $\mathbf{H}^2$ is the set of real-valued predictable processes $\xi$ with $$\Vert \xi \Vert^2_{\mathbf{H}^2} := E \left[ \int_0^T \vert \xi_t \vert^2 dt \right] <\infty.$$ 
\item[•] $\mathbf{S}^{2,p}$ is the vector space of real-valued predictable (not necessarily cadlag) processes $\xi$ such that 
$$ \vertiii{\xi}^2_{\mathbf{S}^{2,p}} := E[ \esssup_{\tau \in \mathcal{T}^p_0} \vert\xi_{\tau} \vert^2 ]<\infty.$$ 
 The mapping $\vertiii{.}_{\mathbf{S}^{2,p}}$ is a norm on the space $\mathbf{S}^{2,p}$. Moreover, the space $\mathbf{S}^{2,p}$ endowed with this norm is a Banach space. This follows by using similar arguments as in the proof of Proposition 2.1 in ~\cite{Reflected BSDEs when the obstacle is not right-continuous and optimal stopping}.

\item[•] $\mathbf{M}^2$ is the set of square integrable martingales $M=(M_t)_{t\in [0,T]}$ with $M_0=0$. \\
We can endow $\mathbf{M}^2$ with the norm $$ \Vert M \Vert_{\mathbf{M}^2} := E \lbrace M_T^2 \rbrace^{\frac{1}{2}}.$$
This space equipped with the scalar product $$(M,N)_{\mathbf{M}^2}=E \lbrace M_T N_T \rbrace =E \lbrace <M,N>_T \rbrace= E\lbrace [M,N]_T \rbrace, \,\,\,\, \text{for}\,\,\,\,  M,N \in \mathbf{M}^2,$$
 is an Hilbert space. Note that each martingale $M$ in $\mathbf{M}^2$ has a cadlag version (cf. \cite{Probabilites et Potentiel2} Theorem 4.VI, p. 76).
\item[•] $\mathbf{M}^{2,\bot}$ is the subspace of martingales $N \in \mathbf{M}^2$ satisfying $<N,W>_. =0$.
\end{enumerate}
\end{notation}

\begin{remark}
The condition $<N,W>_. =0$ expresses the orthogonality of $N$ (in the sense of the scalar product $(.,.)_{\mathbf{M}^2}$) with respect to all stochastic integrals of the form $\int_0^. h_s dW_s$, where $h \in \mathbf{H}^2$ (cf. e.g., \cite{Protter} IV. 3 Lemma 2, p. 180). 
\end{remark} 

  For a ladlag process $\phi$, we denote by $\phi_{t+}$ and $\phi_{t-}$ the right-hand and left-hand limit of $\phi$ at time $t$. We denote by $\Delta \phi_t := \phi_t - \phi_{t-}$ the size of left jump of $\phi$ at time $t$, and by $\Delta_+ \phi_t := \phi_{t+} - \phi_t$ the size of right jump of $\phi$ at time $t$.
  
\begin{definition}
A predictable process $(\phi_t)$ is said to be left upper-semicontinuous (resp. left lower-semicontinuous) along predictable stopping times if for each $\tau\in \mathcal{T}_0^p$, for each nondecreasing sequence of stopping times $(\tau_n)$ such that $\tau_n \uparrow \tau$ a.s., we have $\phi_\tau \geq \limsup_{n} \phi_{\tau_n}$ (resp. $\phi_\tau \leq \limsup_{n} \phi_{\tau_n})$ a.s.
\end{definition}

The following orthogonal decomposition property of martingales in $\mathbf{M}^2$ can be found in  ~\cite{Jacod and Shiryaev} (Chapter III, Lemma 4.24, p. 185):
\begin{lemma} \label{orthogonal decomposition}
For each $M\in \mathbf{M}^2$, there exists a unique couple $(Z,N)\in \mathbf{H}^2 \times \mathbf{M}^{2,\bot} $ such that
\begin{equation}
M_t= \int_0^t Z_s dW_s + N_t, \,\,\,\, \forall t \in [0,T] \,\,\,\, a.s.
\end{equation}
\end{lemma}

\begin{definition}
A predictable process $Y=(Y)_{t\in [0,T]}$ is a predictable strong supermartingale if 
\begin{enumerate}
\item For every bounded predictable time $\tau$, $Y_\tau$ is integrable.
\item For every pair of predictable times $S$, $\tau$ such that $S \leq \tau$,
 \begin{equation}
 Y_{S} \geq \mathbb{E}[Y_\tau \vert \mathcal{F}_{S^-}] \,\,\,\, a.s.
\end{equation}   
\end{enumerate}
\end{definition}



 The following Theorem can be found in ~\cite{Probabilites et Potentiel1} (Theorem 86, p. 220).
 
\begin{theorem}(Section theorem) \label{section}
Let $X=(X_t)$ and $Y=(Y_t)$ be two optional (resp. predictable) processes. If for every finite
stopping time $\tau$ one has, $X_\tau = Y_\tau $, then the processes $(X_t)$ and $(Y_t)$ are indistinguishable.
\end{theorem}

\begin{definition}(Driver, Lipschitz driver). A function $g$ is said to be a driver if 
\begin{enumerate}
\item[(i)] $g:\Omega \times [0,T] \times \mathbf{R}^2 \rightarrow \mathbf{R}$\\
         $(\omega,t,y,z)  \mapsto g(\omega,t,y,z)$ is $\mathcal{P} \otimes \mathcal{B}(\mathbf{R^2})-$ measurable,
\item[(ii)]$g(.,0,0) \in \mathbf{H}^2$.
\end{enumerate}
A driver $g$ is called a Lipschitz driver if moreover there exists a constant $K> 0$ such that $dP \otimes dt$-a.s., for each $(y_1,z_1)\in \mathbf{R}^2$, $(y_2,z_2)\in \mathbf{R}^2$, 
$$\vert g(w,t,y_1,z_1)-g(w,t,y_2,z_2) \vert \leq K ( \vert y_1-y_2 \vert - \vert z_1-z_2 \vert).$$
\end{definition}

We recall the definition of mutually singular random measures associated with non decreasing cadlag predictable processes from ~\cite{Generalized Dynkin Games and Doubly reflected BSDEs with jumps} (Definition 2.3., p. 5). 

\begin{definition}
Let $A=(A_t)_{0 \leq t\leq T}$ and $A'=(A'_t)_{0 \leq t\leq T}$ be two real-valued predictable non-decreasing cadlag processes with $A_0=0$, $A'_0=0$, $E[A_T]<\infty$ and $E[A'_T]<\infty$. We say that the random measures $dA_t$ and $dA'_t$ are \textit{mutually singular} (in a probabilistic sense), and we write $dA_t \perp dA'_t$, if there exists $D\in \mathcal{O}$ such that: 
\begin{equation}
E \left[ \int_0^T \textbf{1}_{D^c} dA_t \right]= E \left[ \int_0^T \textbf{1}_{D} dA'_t \right]=0,
\end{equation}
which can also be written as $ \int_0^T \textbf{1}_{D^c_t} dA_t= \int_0^T \textbf{1}_{D_t} dA'_t =0  $ a.s., where for each $t\in [0,T]$, $D_t$ is the section at time $t$ of $D$, that is, $D_t:= \lbrace \omega\in \Omega, (\omega,t)\in D \rbrace$. 
\end{definition}

\begin{definition}(Predictable admissible obstacles)
 Let $\xi=(\xi_t)_{t\in[0,T]}$ and $\zeta=(\zeta_t)_{t\in[0,T]}$ be two processes in $\mathbf{S}^{2,p}$ such that $\xi_t \leq \zeta_t$, $0\leq t \leq T$, a.s. and $\xi_T = \zeta_T$ a.s. A pair of processes $(\xi,\zeta)$ satisfying the previous properties will be called a pair of predictable admissible obstacles, or a pair of predictable admissible barriers.
\end{definition}

\subsection{Doubly RBSDE whose obstacles are predictable in the case of non quasi-left continuous filtration}
   \hspace*{0.5 cm} Let $g$ be a driver. Let $(\xi,\zeta)$ be a pair of ladlag predictable admissible obstacles.

\begin{definition} \label{def1}
A process $(Y,Z,M,A,B,A',B') \in \mathbf{S}^{2,p} \times \mathbf{H}^2 \times \mathbf{M}^{2,\perp} \times (\mathbf{S}^{2,p})^2 \times (\mathbf{S}^{2,p})^2 $ is said to be a solution to the predictable doubly RBSDE with parameters $(g,\xi,\zeta)$, where $g$ is a driver and $(\xi, \zeta)$ is a pair of predictable admissible obstacles, if 
\begin{multline}
Y_\tau=\xi_T + \int_\tau^T g(s,Y_s,Z_s) ds - \int_\tau^T Z_s dW_s - (M_{T^-}-M_{\tau^-}) +  A_T-A_\tau \\ - (A'_T-A'_\tau) + B_{T^-}-B_{\tau^-}-(B'_{T^-}-B'_{\tau^-}), \,\,\,\, \text{for all} \,\,\, \tau\in \mathcal{T}_0^p, \label{eq*}
\end{multline}
with 
\begin{enumerate}
\item[(i)] $ \xi_\tau \leq Y_\tau \leq \zeta_\tau$  a.s. for all $\tau\in \mathcal{T}_0^p$,
\item[(ii)] $A$ and $A'$ are nondecreasing right-continuous predictable processes with $A_0=A'_0=0$ and such that 
\begin{equation}
\int_0^T \textbf{1}_{\lbrace Y_{t^-}>\xi_{t^-} \rbrace} dA_t=0 \,\,\, \text{a.s.    and    } \int_0^T \textbf{1}_{\lbrace Y_{t^-}<\zeta_{t^{-}} \rbrace} dA'_t =0  \,\,\, \text{a.s.} \label{sko1}
\end{equation}
\item[(iii)] $B$ and $B'$ are nondecreasing, right-continuous predictable purely discontinuous processes with $B_{0^-}=B'_{0^-}=0$,
\begin{equation}
(Y_\tau-\xi_\tau)(B_\tau -B_{\tau^-})=0 \,\,\, \text{and} \,\,\, (Y_\tau-\zeta_\tau)(B'_\tau -B'_{\tau^-})=0 \,\,\, \text{a.s. for all $\tau \in \mathcal{T}_0^p$,} \label{sko2}
\end{equation}
\item[(iv)] $dA_t \perp dA'_t$ and $dB_t \perp dB'_t$. \label{mutualy singular}
\end{enumerate}
\end{definition}

 The equations (\ref{sko1}) and (\ref{sko2}) are called minimality conditions or Skorohod conditions.  
 
 \begin{remark}
 Note that a process $(Y,Z,M,A,B,A',B') \in \mathbf{S}^{2,p} \times \mathbf{H}^2 \times \mathbf{M}^{2,\perp} \times (\mathbf{S}^{2,p})^2 \times (\mathbf{S}^{2,p})^2 $ satisfies equation (\ref{eq*}) in the above definition if and only if, almost surely, for all $t$ in $[0,T]$, 
 \begin{multline} \label{eq**}
Y_t=\xi_T + \int_t^T g(s,Y_s,Z_s) ds - \int_t^T Z_s dW_s - (M_{T^-}-M_{t^-}) +  A_T-A_t \\- (A'_T-A'_t) +B_{T^-}-B_{t^-}-(B'_{T^-}-B'_{t^-}).
 \end{multline}
 \end{remark}

\begin{remark}
Note that, if we abandon the mutually singularity constraint (iv), the processes $A$ and $A'$ (resp. $B$ and $B'$) can increase at the same time on $\lbrace \xi_{t^-}=\zeta_{t^-} \rbrace$ (resp. on $\lbrace \xi_t=\zeta_t \rbrace$). This constraint permits us to obtain the uniqueness of the nondecreasing processes $A, A', B$ and $B'$ without the usual strict separability condition $\xi <\zeta$ (cf. e.g., \cite{El Asri and al.}). 
\end{remark}

\begin{remark}
If we rewrite the equation (\ref{eq**}) forwardly, we obtain $-(Y_\tau - Y_{\tau^{-}})=\Delta A_\tau - \Delta A'_\tau$ a.s. for each predictable stopping time $\tau \in \mathcal{T}_0^p$. Hence, the left jump process $ \Delta Y$ satisfies: $\Delta Y \equiv \Delta A' - \Delta A$. Indeed, the processes $Y$, $A$ and $A'$ are predictable. Thus, $\Delta Y$, $\Delta A$ and $\Delta A'$ are also predictable. The result follows from an application of Theorem \ref{section}.
This together with the condition $dA_t \perp dA'_t$ ensures that for each $\tau \in \mathcal{T}_0^p$, $\Delta A_\tau= (\Delta Y_{\tau})^-$ a.s., and $\Delta A'_\tau=(\Delta Y_{\tau})^+$ a.s. \label{remark2}
\end{remark}

\begin{remark}
We restrict our attention to the fact that the term $M_{-}$ in the equation (\ref{eq*}) satisfied by $Y$ is not a martingale but the predictable projection of the martingale $M$, which differs from the case of the doubly reflected BSDEs associated with optional obstacles (cf. e.g., \cite{Doubly Reflected BSDEs and E-Dynkin games}).

\end{remark}

\begin{remark}
We note also that, for each $\tau \in \mathcal{T}_0^p$, $\Delta B_\tau - \Delta B'_\tau = -({}^pY_{\tau}^+ - Y_\tau)$. This equality follows from the fact that $M$ is a right-continuous martingale, hence, $^pM=M_{-}$ and the fact that $A,A',B$ and $B'$ are predictable processes. \\
This together with the condition $dB_t \perp dB'_t$ ensures that for each $\tau \in \mathcal{T}_0^p$, $\Delta B_\tau= ({}^pY_{\tau}^+ - Y_\tau)^-$ a.s., and $\Delta B'_\tau=({}^pY_{\tau}^+ - Y_\tau)^+$ a.s. \label{remark4}
\end{remark}

\begin{remark}(Quasi-left-continuous filtration)
In the case where the filtration is quasi-left-continuous, martingales have only totally inaccessible jumps. Hence, from the equation (\ref{eq*}) we can see that for each $\tau \in \mathcal{T}_0^p$, $\Delta B_\tau - \Delta B'_\tau = - \Delta_+Y_\tau$. This together with the condition $dB_t \perp dB'_t$, ensures that for each $\tau \in \mathcal{T}_0^p$, $\Delta B_\tau = (\Delta_+Y_\tau)^-$ a.s., and $\Delta B'_\tau = (\Delta_+Y_\tau)^+$ a.s.\\
We point out that in the case of a general filtration, this property does not necessarily hold. Indeed, by equation (\ref{eq*}) for each  $\tau \in \mathcal{T}_0^p$, we have $ \Delta_+Y_\tau = \Delta M_\tau + \Delta B'_\tau - \Delta B_\tau$ a.s., and $\Delta M_\tau$ is here not necessarily equal to $0$, since in this case martingales may admit
jumps at some predictable stopping times. 
\end{remark}

\begin{proposition}
Let $g$ be a driver and $(\xi,\zeta)$ be a pair of admissible obstacles. Let $(Y,Z,M,A,B,A',B') $ be a solution to the predictable doubly reflected BSDE with parameters $(g,\xi,\zeta)$. 
\begin{enumerate}
\item[(i)] For each $\tau \in \mathcal{T}_0^p$, we have 
\begin{equation}
Y_\tau= ({}^pY_\tau^+ \vee \xi_\tau  )\wedge \zeta_\tau \,\,\, \text{a.s.}
\end{equation}
\item[(ii)] If $\xi$ (resp. $\zeta$) is right continuous, then $B=0$ (resp. $B'=0$).
\item[(iii)] If $\xi$ (resp. $\zeta$) is left upper-semicontinuous (resp. left lower-semicontinuous) along predictable stopping times, then the process $A$ (resp. $A'$) is continuous. 
\end{enumerate}
\end{proposition}

\begin{proof}
Let $\tau \in \mathcal{T}_0^p$. By Remark \ref{remark4}, we have $\Delta B_\tau= ({}^pY_{\tau}^+ - Y_\tau)^-$ and $\Delta B_\tau= ({}^pY_{\tau}^+ - Y_\tau)^+$ a.s. Since $B$ and $B'$ satisfy the Skorokhod condition (\ref{sko2}), we get
\begin{equation*}
({}^pY_{\tau}^+ - Y_\tau)^- = \textbf{1}_{\lbrace Y_\tau = \xi_\tau \rbrace} ({}^pY_{\tau}^+ - Y_\tau)^- \,\,\,\text{and} \,\,\, ({}^pY_{\tau}^+ - Y_\tau)^+ = \textbf{1}_{\lbrace Y_\tau = \zeta_\tau \rbrace} ({}^pY_{\tau}^+ - Y_\tau)^+.
\end{equation*}
Hence, on the set $\lbrace \xi_\tau <Y_\tau<\zeta_\tau \rbrace$, we have $Y_\tau={}^pY_{\tau}^+$ a.s., which implies that $( {}^pY_\tau^+ \vee \xi_\tau  )\wedge \zeta_\tau = Y_\tau$ a.s. Now, on the set $\lbrace \xi_\tau <Y_\tau=\zeta_\tau \rbrace$, we have
$({}^pY_{\tau}^+ - Y_\tau)^-=0$ a.s., which gives ${}^pY_{\tau}^+ \geq Y_\tau = \zeta_\tau > \xi_\tau$ a.s. Hence,
$({}^pY_\tau^+ \vee \xi_\tau  )\wedge \zeta_\tau ={}^pY_\tau^+ \wedge \zeta_\tau =\zeta_\tau=Y_\tau$ a.s. \\
Similarly, on the set $\lbrace \xi_\tau =Y_\tau<\zeta_\tau \rbrace$, we have $({}^pY_\tau^+ \vee \xi_\tau  )\wedge \zeta_\tau = Y_\tau$ a.s. The first assertion thus holds.

We now prove the second assertion. Suppose that $\xi$ is right-continuous. Let $\tau \in \mathcal{T}_0^p$. We show that $\Delta B=0$ a.s. From the remarks above, we have 
\begin{equation*}
\Delta B_\tau = \textbf{1}_{\lbrace Y_\tau = \xi_\tau \rbrace} ({}^pY_{\tau}^+ - Y_\tau)^-= \textbf{1}_{\lbrace Y_\tau = \xi_\tau \rbrace} ({}^pY_{\tau}^+ - \xi_\tau)^- = \textbf{1}_{\lbrace Y_\tau = \xi_\tau \rbrace} ({}^pY_{\tau}^+ - {}^p\xi_{\tau}^+)^- \,\,\,\text{a.s.,}
\end{equation*}
where the last inequality follows from the fact that $\xi$ is right-continuous predictable process. Since $Y \geq \xi$, we derive that, for all $\tau \in \mathcal{T}_0^p$, $\Delta B_\tau=0$ a.s. Since, $B$ is purely discontinuous non decreasing process, null at $0$, it follows that $B=0$. Similarly, it can be shown that if $\zeta$ is right-continuous, then $B'=0$. Hence the second assertion holds. 

It only remains to prove the third assertion. Suppose that $\xi$ is left upper-semicont\\inuous along predictable stopping times. We show $\Delta A_\tau=0$ a.s. Let $\tau \in \mathcal{T}_0^p$. By Remark \ref{remark2} and condition (\ref{sko1}), we derive that 
\begin{equation}
\Delta Y_\tau= - \Delta A_\tau + \Delta A'_\tau = - \Delta A_\tau \textbf{1}_{\lbrace Y_{\tau^-}=\xi_{\tau^-}\rbrace \cap K} + \Delta A'_\tau  \textbf{1}_{\lbrace Y_{\tau^-}=\zeta_{\tau^-} \rbrace \cap K'} \,\,\, \text{a.s.,} \label{deltaY}
\end{equation}
where $K= \lbrace \Delta A_\tau >0 \rbrace$ and $K'= \lbrace \Delta A'_\tau >0 \rbrace$. Note that $K$ and $K'$ belong to $\mathcal{F}_{\tau^-}$. Since $dA_t \perp dA'_t$ , we get $K \cap K' = \emptyset$ a.s. Hence, on ${\lbrace Y_{\tau^-}=\xi_{\tau^-}\rbrace \cap K}$, we have 
\begin{equation*}
\Delta Y_\tau= - \Delta A_\tau \leq 0 \,\,\,\,\,\, \text{a.s.} 
\end{equation*}
Since $\xi$ is left-u.s.c along stopping times, we hence drive that on ${\lbrace Y_{\tau^-}=\xi_{\tau^-}\rbrace \cap K}$, we have 
\begin{equation*}
\xi_{\tau^-} \leq \xi_\tau \leq Y_\tau \leq Y_{\tau^-} 	\,\,\,\,\,\, \text{a.s.,}
\end{equation*}
and the inequalities are even equalities (still on the set ${\lbrace Y_{\tau^-}=\xi_{\tau^-}\rbrace \cap K}$). Hence, $\Delta Y_\tau=0$ a.s. on ${\lbrace Y_{\tau^-}=\xi_{\tau^-}\rbrace \cap K}$. By (\ref{deltaY}), we derive that $\Delta A_\tau=0$ a.s. This equality being true for every predictable stopping time $\tau \in \mathcal{T}_0^p$, it follows that $A$ is continuous. Similarly, it can be shown that if $\zeta$ is left lower-semicontinuous along predictable stopping times, then $A'$ is continuous, and the proof is complete.
\end{proof}

\begin{remark}
If $\xi$ and $\zeta$ are predictable obstacles such that $\xi$ is l.u.s.c. and $\zeta$ is l.l.s.c. along stopping times. Then $Y$ is left-continuous. This is a direct consequence of the third assertions of the proposition above and Remark  \ref{remark2}. 
\end{remark}

  We now give a necessary condition for the existence of a solution of the doubly reflected BSDE from Definition \ref{def1}.

\begin{definition}(\textsc{Mokobodzki's condition in the Predictable setting}) \\
Let $(\xi,\zeta) \in \mathbf{S}^{2,p} \times \mathbf{S}^{2,p}$ be a pair of predictable admissible barriers. We say that the pair $(\xi,\zeta)$ satisfies Mokobodzki's condition if there exist two nonnegative predictable strong supermartingales $H^p$ and $\overline{H}^p$ in $ \mathbf{S}^{2,p}$ such that:
\begin{equation}
\xi_t \leq H^p_t- \overline{H}_t^p \leq \zeta_t \,\,\,\,\,\,\,\, 0\leq t \leq T \,\,\,\,\,\, \text{a.s.}
\end{equation}
\end{definition}

\begin{lemma}
 Let $g \in \mathbf{H}^2$. Let $(\xi,\zeta)$ be a pair of predictable admissible barriers. Then, Mokobodzki's condition is a necessary condition for the existence of a solution to the predictable DRBSDE associated with the driver process $(g_t)$ and the barriers $\xi$ and $\zeta$.
\end{lemma}

\begin{proof}
Let $(Y,Z,M,A,B,A',B') \in \mathbf{S}^{2,p} \times \mathbf{H}^2 \times \mathbf{M}^{2,\perp} \times (\mathbf{S}^{2,p})^2 \times (\mathbf{S}^{2,p})^2 $ be a solution to the predictable DRBSDE associated with driver $g$ and the barriers $\xi$ and $\zeta$. Since $Y$ is predictable, we have $E(Y_t \vert \mathcal{F}_{t^-})= Y_t$. From this and equality (\ref{eq**}), it follows that  
\begin{align*}
Y_t &=\mathbb{E}(\xi_T^+ + \int_t^T g^+(s,Y_s,Z_s) ds + A_T-A_t + B_{T^-}- B_{t^-}\vert \mathcal{F}_{t^-}) \\
 &+ \mathbb{E}(\xi_T^- + \int_t^T g^-(s,Y_s,Z_s) ds + A'_T-A'_t + B'_{T^-}- B'_{t^-}\vert \mathcal{F}_{t^-}) 
\\ & - E(M_{T^-} - M_{t^-} \vert \mathcal{F}_{t^-}).
\end{align*}
Otherwise, by noting that $M$ is a (cadlag) uniformly integrable martingale and by applying the predictable stopping theorem in \cite{Nikeghbali} (Theorem 4.5, p 358), we obtain that the last term in the equality above is equal to
zero. Hence, $Y=H^p-\overline{H}^{p}$, where $H^p$ and $\overline{H}^{p}$ are the two non-negative predictable strong supermartingales defined by
\begin{equation*}
H_t^p := \mathbb{E}(\xi_T^+ + \int_t^T g^+(s,Y_s,Z_s) ds + A_T-A_t + B_{T^-}- B_{t^-}\vert \mathcal{F}_{t^-});
\end{equation*}

\begin{equation*}
\overline{H}_t^{p} := \mathbb{E}(\xi_T^- + \int_t^T g^-(s,Y_s,Z_s) ds + A'_T-A'_t + B'_{T^-}- B'_{t^-}\vert \mathcal{F}_{t^-}).
\end{equation*}

Since $\xi \leq Y \leq \zeta$, we get $\xi \leq H^p-H^{'p} \leq \zeta$, which guarantee that the Mokobodzki's condition holds. 
\end{proof}


\section{The (y,z)-independent case}

\hspace*{0.5 cm} Let $(\xi,\zeta)$ be a pair of predictable admissible barriers. Let $g$ be a driver. We assume that $g$ does not depend on $(y, z)$ i.e., $P$-a.s., $g(t, \omega, y, z) \equiv g (t, \omega)$, for any $t, y$ and $z$.\\
 \hspace*{0.5 cm} In this section, we are going to prove the existence and uniqueness, under the above assumptions on $g$, $\xi$ and $\zeta$, of the solution to the predictable doubly RBSDE from Definition \ref{def1}. The idea of the proof is the same as in the paper of Grigorova et al.  \cite{Doubly Reflected BSDEs and E-Dynkin games}, in which the authors proved the results for the doubly RBSDE with a not necessarily continuous optional obstacles and general filtration.

\subsection{Existence of a solution to the predictable DRBSDE with driver process $(g_t)$}

\hspace*{0.5 cm} As a first step, we suppose that there exists a solution of the predictable DRBSDE associated with the driver $g$ and we show that up to the process $E(\xi_T + \int_t^T g_s\, ds \vert \mathcal{F}_{t^-})$, the first component of this solution can be written as the difference of the solutions of two coupled predictable reflected BSDEs. \\

We introduce the following operator:
\begin{definition}(Operator induced by a predictable RBSDE with driver 0)\\ \label{operator}
Let $\xi$ be a process in  $\mathbf{S}^{2,p}$. We denote by $\mathcal{P}re[\xi]$ the first component of the solution to the predictable BSDE from Definition \ref{def2} in the case where the driver is $0$ (see Definition \ref{def2} in the Appendix). 
\end{definition}
\begin{remark}
Note that by Proposition \ref{existence PBSDE} in the Appendix, the operator $\mathcal{P}re : \xi \rightarrow \mathcal{P}re[\xi]$ is well defined on $\mathbf{S}^{2,p}$.
\end{remark}

Let $(Y,Z,M,A,B,A',B') \in \mathbf{S}^{2,p} \times \mathbf{H}^2 \times \mathbf{M}^{2,\perp} \times (\mathbf{S}^{2,p})^2 \times (\mathbf{S}^{2,p})^2 $ be a solution to the predictable DRBSDE associated with driver $g$ and the barriers $\xi$ and $\zeta$.  Let $\tilde{Y}$ the predictable process defined by $\tilde{Y}_t := Y_t - \mathbb{E}[\xi_T + \int_t^T g_s ds \vert \mathcal{F}_{t^-}]$, for all $t\in [0,T]$. From this definition together with equation (\ref{eq**}), we get

\begin{equation}
\tilde{Y}_t= J^{g,p}_t- \bar{J}^{g,p}_t \,\,\,\, \text{for all}\,\,\,\, t\in [0,T] \,\,\,\,\text{a.s.}, \label{J-J'}
\end{equation}

where the  processes $J^{g,p}$ and $\bar{J}^{g,p}$ are defined, for all $t \in [0,T]$, by 
\begin{equation}
J^{g,p}_t := \mathbb{E}[A_T-A_t +B_{T^-}- B_{t^-} \vert \mathcal{F}_{t^-}] \,\,\,\, \text{and} \,\,\,\,
 \bar{J}^{g,p}_t := \mathbb{E}[A'_T-A'_t +B'_{T^-}- B'_{t^-} \vert \mathcal{F}_{t^-}].
\end{equation}
\begin{remark} \label{J,Jbar}
Note that since $A, \,\,A', \,B$ and $B'$ are non-decreasing processes belong to $\mathbf{S}^{2,p}$, $J^{g,p}$ and $\bar{J}^{g,p}$ are two nonnegative predictable strong supermartingales in $\mathbf{S}^{2,p}$ such that $J^{g,p}_T=\bar{J}^{g,p}_T=0$ a.s.
\end{remark}

We introduce the following predictable processes (which also depend on the process $g$):
 \begin{equation}
 \tilde{\xi}_t^{g,p} := \xi_t - \mathbb{E}[\xi_T + \int_t^T g_s ds \vert \mathcal{F}_{t^-}], \,\,\,\,\, \tilde{\zeta}_t^{g,p} := \zeta_t - \mathbb{E}[\zeta_T + \int_t^T g_s ds \vert \mathcal{F}_{t^-}], \,\,\,\,\, 0 \leq t \leq T. 
 \end{equation}
 
\begin{remark} \label{xi,zeta}
Note that since $g \in \mathbf{H}^2$ and $\xi \in \mathbf{S}^{2,p}$, $\tilde{\xi}^{g,p}$ and $\tilde{\zeta}^{g,p}$ belong to $\mathbf{S}^{2,p}$ . Moreover, we have $\tilde{\xi}^{g,p}_T=\tilde{\zeta}^{g,p}_T=0$ a.s. 
\end{remark}

In the following lemma, we prove that the processes $J^{g,p}$ and $\bar{J}^{g,p}$ satisfy a coupled system of predictable reflected BSDEs.

\begin{lemma} \label{implication}
Let $Y \in \mathbf{S}^{2,p}$ be the first component of a solution of the predictable doubly RBSDE with parameters $(g,\xi,\zeta)$, where $g$ is a driver and $(\xi, \zeta)$ is a pair of predictable admissible obstacles. We then have 
$$Y_t= J^{g,p}_t - \bar{J}^{g,p}_t + \mathbb{E}[\xi_T + \int_t^T g_s ds \vert \mathcal{F}_{t^-}], \,\, 0\leq t\leq T, \,\,\, \text{a.s.},$$ 
where the process $J^{g,p}$ and $\bar{J}^{g,p}$ satisfy the following coupled system of predictable reflected BSDEs:
\begin{equation} \label{coupled}
J^{g,p}= \mathcal{P}re[(\bar{J}^{g,p} + \tilde{\xi}^{g,p} ) \textbf{1}_{[0,T)}]; \,\,\,\, \bar{J}^{g,p}= \mathcal{P}re[(J^{g,p}-\tilde{\zeta}^{g,p} ) \textbf{1}_{[0,T)}],
\end{equation}
where $\mathcal{P}re$ is the operator associated to the predictable RBSDE with driver 0 (cf. Definition  \ref{operator}).
\end{lemma}

\begin{proof}
From the definition of $\tilde{Y}$ and equality (\ref{J-J'}), it follows that $$Y_t= J^{g,p}_t - \bar{J}^{g,p}_t + \mathbb{E}[\xi_T + \int_t^T g_s ds \vert \mathcal{F}_{t^-}]\,\, \,\, 0\leq t\leq T \,\,\, \text{a.s.}$$
By Remark \ref{J,Jbar} and \ref{xi,zeta}, the processes $(\bar{J}^{g,p} + \tilde{\xi}^{g,p} ) \textbf{1}_{[0,T)}$ and  $(J^{g,p}-\tilde{\zeta}^{g,p} ) \textbf{1}_{[0,T)}$ belong to $\mathbf{S}^{2,p}$. Moreover, from $\xi \leq Y \leq \zeta$ and the definitions of $\tilde{Y}$, $\tilde{\xi}^{g,p}$, $\tilde{\zeta}^{g,p}$, we get $\tilde{\xi}^{g,p}\leq \tilde{Y} = J^{g,p} - \bar{J}^{g,p} \leq \tilde{\zeta}^{g,p}.$ Therefore, $$J^{g,p} \geq \bar{J}^{g,p} + \tilde{\xi}^{g,p} \,\,\,\, \text{and} \,\,\,\, \bar{J}^{g,p} \geq J^{g,p} - \tilde{\zeta}^{g,p}.$$

By the above, $J^{g,p}$ and $\bar{J}^{g,p}$ are two nonnegative predictable strong supermartingales in $\mathbf{S}^{2,p}$, hence of class $(\mathcal{D})$ (i.e. $\lbrace J^{g,p}_\tau; \tau \in \mathcal{T}_0^p \rbrace$ and $\lbrace \bar{J}^{g,p}; \tau \in \mathcal{T}_0^p \rbrace $ are uniformly integrable). Applying Mertens decomposition for predictable strong supermartingales of class $(\mathcal{D})$ (see \cite{Meyer_Un cours sur les integrales stochastiques}, p. 143), we conclude that 
\begin{align}
J^{g,p}  = N_{t^-} - A^1_t - B^1_{t^-}; \,\,\,\,\,\,\,\, \bar{J}^{g,p} = \bar{N}_{t^-} - A^2_t - B^2_{t^-}, \label{e1}
\end{align}
\,\,where; $N$ and $\bar{N}$ are two cadlag. uniformly integrable martingales,\\
\hspace*{1.1 cm} $A^1$ and $A^2$ two nondecreasing right-continuous predictable processes, \\
\hspace*{1.1 cm} $B^1$ and $B^2$ two nondecreasing right-continuous predictable purely disconti- \\
\hspace*{1.1 cm} nuous processes.\\
Since otherwise, 
\begin{align*}
J^{g,p}_t = \mathbb{E}[A_T+ B_{T^-} \vert \mathcal{F}_{t^-}] -A_t - B_{t^-}; \\
\bar{J}^{g,p}_t = \mathbb{E}[A'_T +B'_{T^-} \vert \mathcal{F}_{t^-}]-A'_t- B'_{t^-},
\end{align*}
the uniqueness of Mertens decomposition implies that $A^1 \equiv A$, $B^1 \equiv B$, $A^2 \equiv A'$ and $B^2 \equiv B'$. \\
Now, by the orthogonal decomposition property of martingales in $\mathbf{M}^2$ (Lemma \ref{orthogonal decomposition}), there exist 
 $(L, M^1), \, (\bar{L}, M^2) \in \mathbf{H}^2 \times \mathbf{M}^{2,\perp} $ such that
\begin{align}
 N_t= \int_0^t L_s dW_s + M^1_t; \,\,\,\,
\bar{N}_t= \int_0^t \bar{L}_s dW_s + M^2_t, \label{e2}
\end{align}
Therefore, combining equation (\ref{e1}) with (\ref{e2}), we get
\begin{align}
J^{g,p}_t & = - \int_t^T L_s dW_s -( M^1_T - M^1_{t^-}) + A_T-A_t + B_{T^-} - B_{t^-} ; \label{eq J^p} \\
\bar{J}^{g,p}_t & = - \int_t^T \bar{L}_s dW_s -(M^2_T - M^2_{t^-}) + A'_T - A'_t +  B'_{T^-} - B'_{t^-}. \label{eq Jbar^p}
\end{align}
Next, we have $Y-\xi = \tilde{Y} - \tilde{\xi}^{g,p} = J^{g,p} - \bar{J}^{g,p} - \tilde{\xi}^{g,p}$. By the Skorokhod condition (\ref{sko2}) satisfied by $B$, we get 
\begin{equation}
\Delta B_\tau (J^{g,p}_\tau - (\bar{J}^{g,p}_\tau + \tilde{\xi}^{g,p}_\tau))=0 \,\,\,\, \text{a.s. for all } \tau \in \mathcal{T}_0^p.
\end{equation}
We also have $\lbrace Y_{t^-} > \xi_{t^-} \rbrace = \lbrace J^{g,p}_{t^-} > \bar{J}^{g,p}_{t^-} + \tilde{\xi}^{g,p}_{t^-} \rbrace$. Hence, the skorokhod condition (\ref{sko2}) satisfied by $A$ can be expressed in the form:
\begin{equation}
\int_0^T \textbf{1}_{ \lbrace J^{g,p}_{t^-} > \bar{J}^{g,p}_{t^-} + \tilde{\xi}^{g,p}_{t^-} \rbrace} dA_t =0 \,\,\,\, \text{a.s.}
\end{equation}
We conclude that $(J^{g,p}, L, M^1, A, B)$ is the solution of the predictable reflected BSDE associated with the driver $0$ and the barrier $(\bar{J}^{g,p} + \tilde{\xi}^{g,p})\textbf{1}_{[0,T)}$\footnote{Note that this barrier is equal to $E(A'_T-A'_t + C'_{T^-} - C'_{t^-} \vert \mathcal{F}_{t^-}) - \tilde{\xi}^{g,p}_t$ if $t<T$, and $0 $ if $t=T$. \label{refnote}}. \\
We prove similarly that  $(\bar{J}^{g,p}, \bar{L}, M^2, A', B')$ is the solution of the predictable reflected BSDE associated with the driver $0$ and the barrier $(J^{g,p} - \tilde{\zeta}^{g,p})\textbf{1}_{[0,T)}$ \footnote{Note that this barrier is equal to $E(A_T-A_t + C_{T^-} - C_{t^-} \vert \mathcal{F}_{t^-}) - \tilde{\zeta}^{g,p}_t$ if $t<T$, and $0 $ if $t=T$. \label{refnote}}.
This completes the proof.
\end{proof}

This lemma shows that the existence of the solution to the predictable DRBSDE with parameters $(g, \xi, \zeta)$, implies the existence of the solution to the coupled system of predictable RBSDEs  (\ref{coupled}). In the following proposition, we prove that this can be seen as an equivalence relation.
\begin{proposition} \label{equivalent proposition}Let $g \in \mathbf{H}^2$. Let $\xi$ and $\zeta$ be two reflected admissible obstacles in $\mathbf{S}^{2,p}$. The following assertions are equivalent:
\begin{enumerate}
\item[(i)] The predictable DRBSDE (\ref{eq*}) with driver process g(t) has a solution. 
\item[(ii)] There exist two predictable processes $J^p_. \in \mathbf{S}^{2,p}$ and $\bar{J}^p_. \in \mathbf{S}^{2,p}$ satisfying the coupled system of predictable DRBSDEs: 
\begin{equation} \label{coupled*}
J^p= \mathcal{P}re[(\bar{J}^p + \tilde{\xi}^{g,p} ) \textbf{1}_{[0,T)}]; \,\,\,\,\,\,\, \bar{J}^p= \mathcal{P}re[(J^p -\tilde{\zeta}^{g,p} ) \textbf{1}_{[0,T)}],
\end{equation}
where $\tilde{\xi}^{g,p}$ and $ \tilde{\zeta}^{g,p}$ as above. 
\end{enumerate}
In this matter, the predictable process $Y$ defined by 
\begin{equation} \label{Y}
Y_t := J^p_t - \bar{J}^p_t + E[ \xi_T + \int_t^T g_s ds  \vert \mathcal{F}_{t^-}], \,\, 0 \leq  t \leq T \,\, \text{a.s.}
\end{equation}
gives the first component of the solution to the predictable DRBSDE. 
\end{proposition}

\begin{remark}
Note that the process $(J^p_t - \bar{J}^p_t + E[ \xi_T + \int_t^T g_s ds  \vert \mathcal{F}_{t^-}])_{t\leq T}$ is ladlag. Indeed, by definition of the operator $\mathcal{P}re[.]$ and Proposition \ref{existence PBSDE} it follows that the processes $J^p$ and $\bar{J}^p$ are predictable strong supermartingales. Accordingly, both are ladlag (cf. \cite{Probabilites et Potentiel2} Theorem 4, Appendix 1, p. 406), and the remark follows.
\end{remark}
\begin{proof}
The implication $(i) \Rightarrow (ii)$, has been proved in Lemma \ref{implication}. Let us prove $(ii) \Rightarrow (i)$. The steps of the proof are similar of those used in the literature (see eg., \cite{Generalized Dynkin Games and Doubly reflected BSDEs with jumps}, \cite{Doubly Reflected BSDEs and E-Dynkin games}). Let $(J^p, L, M^1,A,B)$ and $(\bar{J}^p, \bar{L}, M^2,A',B') $  be two solutions in $ \mathbf{S}^{2,p} \times \mathbf{H}^2 \times \mathbf{M}^{2,\perp} \times (\mathbf{S}^{2,p})^2 \times (\mathbf{S}^{2,p})^2$ of the coupled system (\ref{coupled*}) . We define the process $Y$ as in (\ref{Y}). \\
\hspace*{0.5 cm} By assumptions, the processes $J^p$ and $\bar{J}^p$ belong to $\mathbf{S}^{2,p}$. Hence, the difference $J^p - \bar{J}^p \in \mathbf{S}^{2,p}$ and thus the process $Y$ is well defined. Further, since $J^p_T=\bar{J}^p_T=0$ a.s., we get $Y_T=\xi_T$ a.s. By the formulation of the coupled system (\ref{coupled*}), we get $J^p_\tau \geq \bar{J}^p_\tau + \tilde{\xi}^{g,p}_\tau$ and $\bar{J}^p_\tau \geq J^p_\tau - \tilde{\zeta}^{g,p}_\tau$ a.s. for all $\tau \in \mathcal{T}_0^p$. Then, we derive that $ \xi_\tau \leq Y_\tau \leq \zeta_\tau$ a.s. for all $\tau \in \mathcal{T}_0^p$. \\
 \hspace*{0.5 cm} We note also that equations (\ref{eq J^p}) and (\ref{eq Jbar^p}), with $J^p$ and $\bar{J}^p$ in place of $J^{g,p}$ and $\bar{J}^{g,p}$, still hold. Otherwise, we can see that the process $(E(\xi_T + \int_t^T g_s \vert \mathcal{F}_{t^-}))_{t\in [0,T]}$ coincides with the first component of the solution to the (non-reflected) predictable BSDE with terminal condition $\xi_T$ and driver $g$ (cf. \cite{Bouhadou} p. 2). Thus, there exist $(L', \bar{M}) \in \mathbf{H}^2 \times \mathbf{M}^{2,\perp}$ such that:
 \begin{equation}
 E(\xi_T + \int_t^T g_s ds \vert \mathcal{F}_{t^-}) = \xi_T + \int_t^t g_s ds - \int_t^T L'_s dW_s - (\bar{M}_{T^-} - \bar{M}_{t^-}).
\end{equation}  
From this, together with (\ref{Y}), and equations (\ref{eq J^p}) and (\ref{eq Jbar^p}) for $J^p$ and $\bar{J}^p$, we get  
\begin{multline}
Y_t= \xi_T + \int_t^T g_s ds - \int_t^T Z_s dW_s -( M_{T^-}- M_{t^-}) + A_T- A_t - (A'_T- A'_t) \\ + B_{T^-} - B_{t^-}  - (B'_{T^-} - B'_{t^-}),
\end{multline}
where $Z:= L-\bar{L} + L'$, $M=M^1-M^2 + \bar{M}$. \\
We now prove that $A$, $A'$, $B$ and $B'$ satisfying the Skorokhod conditions (\ref{sko1}) and (\ref{sko2}). From the above, it follows that the processes $A$, $B$ (resp. $A'$, $B'$) satisfy the Skorokhod conditions for the predictable RBSDEs. Accordingly, we have: for all $\tau \in \mathcal{T}^p_0$, $\Delta A_\tau = \textbf{1}_{\lbrace J^p_{\tau^-} = \bar{J}^p_{\tau^-} + \tilde{\xi}^{g,p}_{\tau^-}\rbrace} \Delta A_\tau=$ a.s., $\Delta B_{\tau} = \textbf{1}_{\lbrace J^p_\tau = \bar{J}^p_\tau + \tilde{\xi}^{g,p}_\tau \rbrace} \Delta B_{\tau}$ a.s. and $\int_0^T \textbf{1}_{\lbrace J^p_{t} > \bar{J}^p_{t} + \tilde{\xi}^{g,p}_{t} \rbrace} dA^c_t=0$ a.s. (see \cite{Bouhadou} Lemma 7, p. 12). Similar conditions hold for $A'$ and $B'$. By the definition of $Y$ and $\tilde{\zeta}^{g,p}$, we get 
\begin{align*}
\lbrace J^p_{t} > \bar{J}^p_{t} + \tilde{\xi}^{g,p}_{t} \rbrace &= \lbrace Y_{t} > \xi_{t} \rbrace \\
 \lbrace J^p_\tau = \bar{J}^p_\tau + \tilde{\xi}^{g,p}_\tau \rbrace &= \lbrace Y_\tau=\xi_\tau \rbrace \\ 
  \lbrace J^p_{\tau^-} = \bar{J}^p_{\tau^-} + \tilde{\xi}^{g,p}_{\tau^-}\rbrace &= \lbrace Y_{\tau^-} = \xi_{\tau^-} \rbrace. 
\end{align*}
This, together with the previous observation implies $\int_0^T \textbf{1}_{\lbrace Y_{t^-} > \xi_{t^-} \rbrace} dA_t =0$ a.s. and $\Delta B_{\tau} = \textbf{1}_{\lbrace Y_\tau=\xi_\tau \rbrace}  \Delta B_{\tau}$ a.s. for all $\tau \in \mathcal{T}^p_0$. By the same manner, we get $\int_0^T \textbf{1}_{\lbrace Y_{t^-} > \xi_{t^-} \rbrace} dA_t =0$ a.s. and $\Delta B_{\tau} = \textbf{1}_{\lbrace Y_\tau=\xi_\tau \rbrace}  \Delta B_{\tau}$ a.s. for all $\tau \in \mathcal{T}^p_0$. The only point remaining concerns the behavior of the random measures $dA_t$, $dA'_t$ and $dB_t$, $dB'_t$.\\

 If $dA_t \perp dA'_t$ and $dB_t \perp  dB'_t$, then the vector $(Y,Z,M,A,B,A',B')$ is a solution to the predictable doubly RBSDE with parameters $(f,\xi,\zeta)$, which is the desired conclusion. If not, by the canonical decomposition of a rcll predictable process with integrable total variation (see Proposition A.7 in \cite{Generalized Dynkin Games and Doubly reflected BSDEs with jumps}, p. 30), there exist a pair $(C,C')$ (resp. $(D,D')$) of nondecreasing right-continuous predictable processes belonging to $\mathbf{S}^{2,p}$, such that $A-A'= C-C'$ (resp. $B-B'= D-D'$) with $dC_t \perp dC'_t$ (resp. $dD_t \perp dD'_t$). Moreover, $dC_t \ll dA_t$, $dC'_t \ll dA'_t$, $dD_t \ll dB_t$ and $dD'_t \ll dB'_t$. \\
 
 Since otherwise $\int_0^T \textbf{1}_{\lbrace Y_{t^-} > \xi_{t^-} \rbrace} dA_t =0$ a.s., we get $\int_0^T \textbf{1}_{\lbrace Y_{t^-} > \xi_{t^-} \rbrace} dC_t =0$ a.s. Similarly, we obtain $\int_0^T \textbf{1}_{\lbrace Y_{t^-} < \zeta_{t^-} \rbrace} dC'_t =0$ a.s. The processes $C$ and $C'$ thus satisfy the Skorokhod conditions (\ref{sko1}). Therefore, the observation $dD_t \ll dB_t$ implies that $D$ is purely discontinuous and $\Delta D_\tau= \textbf{1}_{\lbrace Y_\tau= \xi_\tau \rbrace} \Delta D_\tau$ a.s. for all $\tau \in \mathcal{T}^p_0$. Similarly, $D'$ is purely discontinuous and $\Delta D'_\tau= \textbf{1}_{\lbrace Y_\tau= \zeta_\tau \rbrace} \Delta D'_\tau$ a.s. for all $\tau \in \mathcal{T}^p_0$. Hence, the processes $D$ and $D'$ thus satisfy the Skorokhod conditions (\ref{sko2}).  We conclude that the process $(Y,Z,M,C,D,C',D')$ is a solution to the predictable doubly RBSDE with parameters $(f,\xi,\zeta)$. Which completes the proof.\\
\end{proof}

\hspace*{0.5 cm} We have thus proved that the existence of a solution to the coupled system (\ref{coupled*}) is a sufficient condition for the existence of a solution to the predictable DRBSDE associated with driver process $(g_t)$. In the following, by constructing a Picard-type iterative procedure, we prove that under Mokobodzki's condition, the coupled system has a solution. \\

Set $\mathbf{J}^{p,0}_.=0$ and $\bar{\mathbf{J}}^{p,0}_.=0$ . We define recursively, for each $n\in \mathbb{N}$, the processes:\footnote{We omit the exponent $g$ in the notation for $\mathbf{J}^{p,n}$ and $\bar{\mathbf{J}}^{p,n}$ for sake of simplicity. \label{refnote}}
\begin{equation} \label{coupled**}
\mathbf{J}^{p,n+1} := \mathcal{P}re[(\bar{\mathbf{J}}^{p,n} + \tilde{\xi}^{g,p} ) \textbf{1}_{[0,T)}]; \,\,\,\,\,\,\, \bar{\mathbf{J}}^{p,n+1} := \mathcal{P}re[(\mathbf{J}^{p,n}  -\tilde{\zeta}^{g,p} ) \textbf{1}_{[0,T)}]
\end{equation}
Since $\tilde{\xi}^{g,p}, \tilde{\zeta}^{g,p} \in  \mathbf{S}^{2,p}$. By induction, one can see that the processes $\mathbf{J}^{p,n}$ and $\bar{\mathbf{J}}^{p,n}$ are well defined.

\begin{lemma}\label{lim}
Let $(\xi,\zeta) \in \mathbf{S}^{2,p} \times \mathbf{S}^{2,p}$ be a pair of predictable admissible barriers satisfies Mokobodzki's condition. The sequences of processes $(\mathbf{J}^{p,n}_.)_{n\in\mathbb{N}}$ and $(\bar{\mathbf{J}}^{p,n}_.)_{n\in\mathbb{N}}$ are nondecreasing. The processes $\mathbf{J}^p_.$ and $\bar{\mathbf{J}}^p_.$ defined by 
\begin{equation} \label{lim}
\mathbf{J}^p_. := \lim_{n \rightarrow +\infty} \mathbf{J}^{p,n}_. \,\,\, \text{and} \,\,\, \bar{\mathbf{J}}^p_. := \lim_{n \rightarrow +\infty} \bar{\mathbf{J}}^{p,n}_.
\end{equation}
 are nonnegative strong supermartingales in $\mathbf{S}^{2,p}$ and satisfying the system (\ref{coupled*}) of coupled predictable RBSDEs.
\end{lemma}

\begin{proof}
See Appendix.
\end{proof}

\begin{proposition}
The processes $\mathbf{J}^p_.$ and $\bar{\mathbf{J}}^p_.$ are minimal in the following sense: if $H$ and $\bar{H}$ are two nonnegative predictable strong supermartingales such that $\tilde{\xi}^{g,p} \leq H-\bar{H} \leq \tilde{\zeta}^{g,p}$, then we have $\mathbf{J}^p \leq H$ and $\bar{\mathbf{J}}^p \leq \bar{H}$. 
\end{proposition}

\begin{proof} Let $H$ and $\bar{H}$ be two nonnegative predictable strong supermartingales such that $\tilde{\xi}^{g,p} \leq H-\bar{H} \leq \tilde{\zeta}^{g,p}$.
We begin by proving recursively that for each $n \in \mathbb{N}$, 
\begin{equation} \label{<<}
\mathbf{J}^{p,n} \leq H  \,\,\,\, \text{and} \,\,\,\,  \bar{\mathbf{J}}^{p,n} \leq \bar{H}.
\end{equation} 
First, we have $\mathbf{J}^{p,0}=0 \leq H$ and $\bar{\mathbf{J}}^{p,0}=0 \leq \bar{H}$. Suppose now that, for some fixed
$n \in \mathbb{N}$, equation (\ref{<<}) holds at rank $n$. By the hypotheses $\bar{H} + \tilde{\xi}^{g,p} \leq H$, we derive that 
$$\bar{\mathbf{J}}^{p,n} + \tilde{\xi}^{g,p} \leq \bar{H} + \tilde{\xi}^{g,p} \leq H.$$
As the operator $\mathcal{P}re$ is nondecreasing, we get  $$\mathbf{J}^{p,n+1}= \mathcal{P}re[\bar{\mathbf{J}}^{p,n} + \tilde{\xi}^{g,p}] \leq \mathcal{P}re[H],$$
since $H$ is a predictable strong supermartingale, by Remark \ref{pre=} in Appendix, we have $\mathcal{P}re[H]= H$. Further, $\mathbf{J}^{p,n+1} \leq H$. Similarly, we get $\bar{\mathbf{J}}^{p,n+1} \leq \bar{H}$, which is the desired conclusion. The proof is completed by letting $n$ tend to $+\infty$ in (\ref{<<}). 
\end{proof}
\vspace*{0.1 cm}

By the previous Lemma and Proposition \ref{equivalent proposition}, we derive the following existence result.

\begin{theorem} \label{existence}
Let $g=(g_t)\in \mathbf{H}^{2,p}$ be a driver process. Let $(\xi,\zeta)$ be a pair of predictable admissible barriers satisfying Mokobodzki's condition. Then, there exists a solution of the predictable doubly RBSDE associated with the driver $g$. The first component of this solution is given by 
\begin{equation}
Y_t := \mathbf{J}^p_t - \bar{\mathbf{J}}^p_t + E[ \xi_T + \int_t^T g_s ds  \vert \mathcal{F}_{t^-}], \,\,\,\, \text{a.s.}
\end{equation}
where $ \mathbf{J}^p$ and $\bar{\mathbf{J}}^p$ are the processes defined in (\ref{lim}).
\end{theorem}

\subsection{Uniqueness of the solution of the predictable DRBSDE with driver process $(g_t)$}

\hspace*{0.5 cm} The proof of the uniqueness of the predictable doubly RBSDE solution, associated with the driver process $(g_t)\in \mathbf{H}^{2,p}$, is based on the following useful results. Let $\beta >0$. We first state some notations: 
\begin{enumerate}
\item[•] For $\phi \in \mathbf{H}^{2,p}$, we define $\Vert \phi \Vert_{\beta}^2 := E[\int_0^T e^{\beta s} \xi_s^2 ds]$.
\item[•] For $\xi \in \mathbf{S}^{2,p}$, we define $ \vertiii{\xi}^2_{\beta} := E[ \esssup_{\tau \in \mathcal{T}^p_0} e^{\tau \beta} \vert\xi_{\tau} \vert^2 ]$.
\item[•] For $M \in \mathbf{M}^2$, $\Vert M \Vert^2_{\mathbf{M}^2_\beta} := E(\int_0^T e^{s\beta} d[M]_s) $.
\end{enumerate}
Note that $\Vert \phi \Vert_{\beta}^2$ (resp. $\vertiii{\xi}^2_{\beta}$, $\Vert M \Vert^2_{\beta, \mathbf{M}^2}$) is a norm on $\mathbf{H}^{2,p}$ (resp. $\mathbf{S}^{2,p}$, $\mathbf{M}^2$) equivalent to the norm  $\Vert \xi \Vert^2_{\mathbf{H}^2}$ (resp. $\vertiii{\xi}^2_{\mathbf{S}^{2,p}}$,  $\Vert M \Vert_{\mathbf{M}^2}^2$).\\

We now give a priori estimate on the norm of the solution, the following lemma will be used in the sequel.  

\begin{lemma}(A priori estimate) \label{lemma uniqueness} Let $(Y,Z,M,A,B,A',B') \in \mathbf{S}^{2,p} \times \mathbf{H}^2 \times \mathbf{M}^{2,\perp} \times (\mathbf{S}^{2,p})^2 \times (\mathbf{S}^{2,p})^2$ (resp. $(\bar{Y}, \bar{Z}, \bar{M}, \bar{A}, \bar{B}, \bar{A}',\bar{B}') \in \mathbf{S}^{2,p} \times \mathbf{H}^2 \times \mathbf{M}^{2,\perp} \times (\mathbf{S}^{2,p})^2 \times (\mathbf{S}^{2,p})^2$) be a solution to the predictable DRBSDE associated with driver $g=(g_t) \in \mathbf{H}^{2,p}$ (resp. $\bar{g}=(\bar{g}_t) \in \mathbf{H}^{2,p}$) and the admissible barriers $\xi$ and $\zeta$. Then, there exists $c>0$ such that for all $\epsilon \geq 0$, for all $\beta >\frac{1}{\epsilon^2}$, we have 
\begin{equation}\label{l1}
\Vert Z- \bar{Z} \Vert^2_{\beta} + \Vert M- \bar{M} \Vert^2_{\mathbf{M}^2_\beta} \leq \epsilon^2 \Vert g- \bar{g} \Vert^2_{\beta}; \\
\end{equation}
and 
\begin{equation} \label{l11}
\vertiii{Y- \bar{Y}}^2_{\beta} \leq 2\epsilon^2(1 + 8c^2) \Vert g- \bar{g} \Vert^2_{\beta}.
\end{equation}
\end{lemma}

\begin{proof} Let $\beta >0$ and  $\epsilon>0$ be such that $\beta >\frac{1}{\epsilon^2}$. We set $\tilde{Y}:= Y-\bar{Y}$, $\tilde{Z}:= Z-\bar{Z}$, $\tilde{M}:= M-\bar{M}$, $\tilde{A}:= A-\bar{A}$, $\tilde{A'}:= A'-\bar{A}'$, $\tilde{B}:= B-\bar{B}$, $\tilde{B'}:= B'-\bar{B}'$ and $\tilde{g}:= g-\bar{g}$. Note that $\tilde{Y}_T := \xi_T- \xi_T=0$. Further, $\tilde{Y}$ can be defined as follow:
\begin{align*}\label{Ytild}
\tilde{Y}_t = \int_t^T \tilde{g}_s ds - \int_t^T \tilde{Z}_s dW_s &-(\tilde{M}_{T^-} - \tilde{M}_{t^-}) 
+ \tilde{A}_T - \tilde{A}_t - (\tilde{A'}_T - \tilde{A'}_t) \\
&+ \tilde{B}_{T^-} - \tilde{B}_{t^-} - (\tilde{B'}_{T^-} - \tilde{B'}_{t^-}), \,\,\,\,\, \text{a.s. for all}  \,\,\, t \in[0,T]. \numberthis
\end{align*}
From (\ref{Ytild}), it is easily seen that $\tilde{Y}$ is an optional strong semimartingale in the vocabulary of \cite{Meyer_Un cours sur les integrales stochastiques} with decomposition
\begin{equation*}
\tilde{Y}_t = \tilde{Y}_0 + \underline{M}_t + \underline{A}_t + \underline{B}_t,
\end{equation*}
where, $\underline{M}_t:= \int_0^t \tilde{Z}_s dW_s + \tilde{M}_{t^-}, \,\,\,$  $\underline{A}_t := - \int_0^t \tilde{g}_s ds -(\tilde{A}_t -\tilde{A'}_t)$ and  $\underline{B}_t:= -\tilde{B}_{t^-} + \tilde{B'}_{t^-}$. Applying the change of variables formula for optional semimartingales (cf.  Theorem 8.2 in \cite{Gal'chouk} and Section 3 in \cite{Lenglart}) or (cf. Corollary A.2 in \cite{Reflected BSDEs when the obstacle is not right-continuous and optimal stopping}) to  $e^{\beta t} \tilde{Y}^2_t$ and using the property $<\tilde{M}^c,W>=0$, we obtain almost surely, for all $t\in [0,T]$, 
\begin{align*}
e^{\beta T} \tilde{Y}^2_T &=  e^{\beta t} \tilde{Y}^2_t + \int_{]t,T]} \beta e^{s\beta} \tilde{Y}_s^2 ds - 2 \int_{]t,T]} e^{\beta s} \tilde{Y}_{s^-} \tilde{g}_s ds + \int_{]t,T]} e^{\beta s} d<\tilde{M}^c>_s \\
&- 2 \int_{]t,T]} e^{s\beta} \tilde{Y}_{s^-} d\tilde{A}_s + 2 \int_{]t,T]} e^{s\beta} \tilde{Y}_{s^-} d\tilde{A'}_s - 2 \int_{]t,T]} e^{s\beta} \tilde{Y}_{s^-} d\tilde{B}_s + 2 \int_{]t,T]} e^{s\beta} \tilde{Y}_{s^-} d\tilde{B'}_s \\
&+ 2 \int_{[t,T[} e^{\beta s} \tilde{Y}_s \tilde{Z}_s dW_s + 2 \int_{[t,T[} e^{\beta s} \tilde{Y}_s d\tilde{M}_s + \int_{]t,T]} e^{\beta s} \tilde{Z}_s^2 ds  \\
&+ \sum_{t<s\leq T} e^{\beta s} (\tilde{Y}_s - \tilde{Y}_{s^-})^2 + \sum_{t\leq s< T} e^{\beta s} (\tilde{Y}_{s^+} - \tilde{Y}_s)^2.
\end{align*}
Since $\tilde{Y}_T=0$, we obtain: almost surely, for all $t\in [0,T]$,
\begin{align*} \label{---}
e^{\beta t} \tilde{Y}^2_t &+ \int_{]t,T]} e^{\beta s} \tilde{Z}_s^2 ds  + \int_{]t,T]} e^{\beta s} d<\tilde{M}^c>_s = - \int_{]t,T]} \beta e^{s\beta} \tilde{Y}_s^2 ds+ 2 \int_{]t,T]} e^{\beta s} \tilde{Y}_{s^-} \tilde{g}_s ds  \\
&+ 2 \int_{]t,T]} e^{s\beta} \tilde{Y}_{s^-} d\tilde{A}_s - 2 \int_{]t,T]} e^{s\beta} \tilde{Y}_{s^-} d\tilde{A'}_s + 2 \int_{]t,T]} e^{s\beta} \tilde{Y}_{s^-} d\tilde{B}_s - 2 \int_{]t,T]} e^{s\beta} \tilde{Y}_{s^-} d\tilde{B'}_s \\
&- 2 \int_{[t,T[} e^{\beta s} \tilde{Y}_s \tilde{Z}_s dW_s - 2 \int_{[t,T[} e^{\beta s} \tilde{Y}_s d\tilde{M}_s  \\
&- \sum_{t<s\leq T} e^{\beta s} (\tilde{Y}_s - \tilde{Y}_{s^-})^2 - \sum_{t\leq s< T} e^{\beta s} (\tilde{Y}_{s^+} - \tilde{Y}_s)^2. \numberthis
\end{align*}
 By applying the inequality $2ab \leq (\frac{a}{\epsilon})^2 + \epsilon^{2} b^2$, valid for all $(a, b) \in \mathbb{R}^2$, to  \textbf{the second term} on the r.h.s. of equality (\ref{---}), we get: a.s. for all $t \in [0, T]$,
\begin{equation*}
- \int_{]t,T]} \beta e^{s\beta} \tilde{Y}_s^2 ds + 2 \int_{]t,T]} e^{\beta s} \tilde{Y}_{s^-} \tilde{g}_s ds \leq (\frac{1}{\epsilon^2}- \beta) \int_{]t,T]} e^{\beta s} \tilde{Y}_{s^-}^2 ds + \epsilon^2 \int_{]t,T]} e^{\beta s} \tilde{Y}_{s^-} \tilde{g}^2_s ds.
\end{equation*} 
 Since $\beta < \frac{1}{\epsilon^2}$, we get: a.s. for all $t \in [0, T]$,
 \begin{equation}  
 \int_{]t,T]} \beta e^{s\beta} \tilde{Y}_s^2 ds + 2 \int_{]t,T]} e^{\beta s} \tilde{Y}_{s^-} \tilde{g}_s ds \leq \epsilon^2 \int_{]t,T]} e^{\beta s} \tilde{g}^2_s ds. 
  \end{equation}
   Our next objective is to show that \textbf{\textcolor{black}{the third term and fourth}} term on the right-hand side of inequality (\ref{---}) is non-positive.  The proof is based on property (\ref{sko1}) and the inequalities $\xi \leq Y \leq \zeta$ and $\xi \leq \bar{Y} \leq \zeta$. Similar arguments apply to prove that \textbf{the fifth and sixth} term  on the right-hand side of inequality (\ref{---}) is non-positive.
   The details, are similar to those in the case of predictable RBSDE with one lower obstacle (cf. proof of Lemma 2 in \cite{Bouhadou}). Hence, equality (\ref{---}) can be put in this form:
\begin{align*} \label{l2}
e^{\beta t} \tilde{Y}^2_t &+ \int_{]t,T]} e^{\beta s} \tilde{Z}_s^2 ds  + \int_{]t,T]} e^{\beta s} d<\tilde{M}^c>_s \leq \epsilon^2 \int_{]t,T]} e^{\beta s}\tilde{g}^2_s ds \\  
&- 2 \int_{[t,T[} e^{\beta s} \tilde{Y}_s \tilde{Z}_s dW_s - 2 \int_{[t,T[} e^{\beta s} \tilde{Y}_s d\tilde{M}_s \\
&- \sum_{t\leq s< T} e^{\beta s} (\tilde{Y}_{s^+} - \tilde{Y}_s)^2.
\numberthis
\end{align*}
We are now in a position to derive an estimates for $\Vert \tilde{Z} \Vert^2_{\beta}$ and $\Vert \tilde{M} \Vert^2_{\mathbf{M}^2_\beta}$. \\ First, by using the fact that $\Delta_+ \tilde{Y} = \Delta \tilde{M} + \Delta \tilde{B'} - \Delta \tilde{B}$, we get:
\begin{align*} \label{l3}
\sum_{t<s\leq T} e^{\beta s} (\Delta \tilde{M}_s)^2 &- \sum_{t\leq s <T} e^{\beta s} (\Delta_+ \tilde{Y})^2 = - \sum_{t\leq s <T} e^{\beta s} (\Delta \tilde{B'}_s - \Delta \tilde{B}_s) \\ &-2 \sum_{t\leq s <T} e^{\beta s} \Delta \tilde{M}_s (\Delta \tilde{B'}_s - \Delta \tilde{B}_s). \numberthis
\end{align*}
Recall that, for all $t\in [0,T]$,  $[\tilde{M}]_t := <\tilde{M}^c,\tilde{M}^c>_t + \sum_{s\leq t} \Delta \tilde{M}^2_s $. This,  together with equality (\ref{l3}), yields 
\begin{align*} \label{l4}
e^{\beta t} \tilde{Y}^2_t &+ \int_{]t,T]} e^{\beta s} \tilde{Z}_s^2 ds  + \int_{]t,T]} e^{\beta s} d[\tilde{M}]_s \leq \epsilon^2 \int_{]t,T]} e^{\beta s}\tilde{g}^2_s ds \\  
&- 2 \int_{[t,T[} e^{\beta s} \tilde{Y}_s \tilde{Z}_s dW_s - 2 \int_{[t,T[} e^{\beta s} \tilde{Y}_s d\tilde{M}_s \\
&- 2\sum_{t\leq s< T} e^{\beta s} \Delta \tilde{M}_s (\Delta \tilde{B'}_s - \Delta \tilde{B}_s).
\numberthis
\end{align*}
 By applying Cauchy-Schwartz inequality, we get 
\begin{equation}
 E \Bigg[\sqrt{\int_0^T e^{2 \beta s} \tilde{Y}^2_s \tilde{Z}^2_s ds} \Bigg]\leq \vertiii{\tilde{Y}}_{\mathbf{S}^{2,p}} \Vert \tilde{Z} \Vert_{2 \beta}  < \infty.
 \end{equation}
Hence, the term $\int_0^t e^{\beta s} \tilde{Y}_s \tilde{Z}_s dW_s$ has zero expectation. The same result hold for the martingale $\int_0^t e^{\beta s} \tilde{Y}_s d\tilde{M}_s$, as
\begin{equation}
 E \Bigg[\sqrt{\int_0^T e^{2\beta s} \tilde{Y}^2_s d[\tilde{M}]_s}\Bigg] \leq \vertiii{\tilde{Y}}_{\mathbf{S}^{2,p}}  \Vert \tilde{M} \Vert_{\mathbf{M}^2_\beta} < \infty.
 \end{equation}
Let us show that $E \big[\sum_{0\leq s< T} e^{\beta s} \Delta \tilde{M}_s (\Delta \tilde{B'}_s - \Delta \tilde{B}_s) \big ]=0$. The process $\tilde{M}$ is a right-continuous uniformaly integrable martingale. Accordingly, for each predictable stopping time $\tau$, we have $E[ \Delta \tilde{M}_\tau \vert \mathcal{F}_{\tau^-}]= 0$ (cf. e.g., Theorem 4.5, p. 358 in ~\cite{Nikeghbali}). Otherwise, $\tilde{B}$ and $\tilde{B'}$ are  predictable, then $(\Delta \tilde{B'}_\tau - \Delta \tilde{B}_\tau)$ is $\mathcal{F}_{\tau^-}$-measurable. This gives: 
$$E[(\Delta \tilde{B'}_\tau - \Delta \tilde{B}_\tau) \tilde{M}_\tau  \vert \mathcal{F}_{\tau^-}]= (\Delta \tilde{B'}_\tau - \Delta \tilde{B}_\tau) E[\tilde{M}_\tau \vert \mathcal{F}_{\tau^-}]=0$$ 
 We thus get $E \big[\sum_{0\leq s< T} e^{\beta s} \Delta \tilde{M}_s (\Delta \tilde{B'}_s - \Delta \tilde{B}_s) \big ]=0$. \\
By taking expectation on both sides of (\ref{l4}) with t=0, we obtain:
\begin{align*} \label{estimate1}
\tilde{Y}^2_0 + \Vert \tilde{Z} \Vert_\beta^2 + \Vert \tilde{M} \Vert_{\mathbf{M}^2_\beta}^2 \leq \epsilon^2 \Vert \tilde{g}\Vert^2_\beta,
\numberthis
\end{align*}
which established the first inequality of Lemma \ref{l1}.\\

What is left is to determine the estimate for $\vertiii{\tilde{Y}}^2_{\beta}$. From inequality (\ref{l4}), we obtain, for all $\tau \in \mathcal{T}^p_0$
\begin{align*} 
e^{\beta \tau} \tilde{Y}^2_\tau &\leq \epsilon^2 \int_0^\tau e^{\beta s}\tilde{g}^2_s ds + 2 \int_0^\tau e^{\beta s} \tilde{Y}_s \tilde{Z}_s dW_s + 2 \int_0^\tau e^{\beta s} \tilde{Y}_s d\tilde{M}_s \\
&- 2 \int_0^T e^{\beta s} \tilde{Y}_s \tilde{Z}_s dW_s + 2 \int_0^T e^{\beta s} \tilde{Y}_s d\tilde{M}_s \,\,\, \text{a.s.}
\numberthis
\end{align*}
By taking first the essential supremum over $\tau \in \mathcal{T}^p_0$, and then the expectation on both sides of the  previous inequality, we obtain
\begin{align*} \label{l5} 
E[\esssup_{\tau \in \mathcal{T}^p_0} e^{\beta \tau} \tilde{Y}^2_\tau ] \leq  \epsilon^2 \Vert \tilde{g}\Vert^2_\beta + 
2 E[\esssup_{\tau \in \mathcal{T}^p_0} \vert \int_0^\tau e^{\beta s} \tilde{Y}_s \tilde{Z}_s dW_s \vert]  + 2 E[\esssup_{\tau \in \mathcal{T}^p_0} \int_0^\tau e^{\beta s} \tilde{Y}_s d\tilde{M}_s \vert].
\numberthis
\end{align*}
By Burknolder-Davis-Gundy inequalities (applied with $p=1$), we get 
\begin{align*}
2 E[\esssup_{\tau \in \mathcal{T}^p_0} \vert \int_0^\tau e^{\beta s} \tilde{Y}_s \tilde{Z}_s dW_s \vert] &\leq 2c E \Bigg[\sqrt{\int_0^T e^{2 \beta s} \tilde{Y}_s^2 \tilde{Z}^2_s ds } \Bigg] \\
&\leq 2c E\Bigg [\sqrt{\esssup_{\tau \in \mathcal{T}^p_0} e^{\beta \tau} \tilde{Y}^2_\tau \int_0^T e^{\beta s} \tilde{Z}^2_s ds} \Bigg] \\
&\leq \frac{1}{4} \vertiii{\tilde{Y}}^2_{\beta} + 4 c^2 \Vert \tilde{Z} \Vert^2_\beta, \numberthis \label{l6}
\end{align*}
where $c$ is a positive "universal" constant (which does not depend on the other
parameters). By using similar arguments, we get 
\begin{align*} \label{l7}
2 E[\esssup_{\tau \in \mathcal{T}^p_0} \int_0^\tau e^{\beta s} \tilde{Y}_s d\tilde{M}_s \vert] &\leq \frac{1}{4} \vertiii{\tilde{Y}}^2_{\beta} + 4 c^2 \Vert \tilde{M} \Vert^2_{\mathbf{M}_\beta^2}. \numberthis
\end{align*}
From the inequalities (\ref{l5}), (\ref{l6}) and (\ref{l7}), we get 
\begin{equation*}
\vertiii{\tilde{Y}}^2_{\beta} \leq 2 \epsilon^2 \Vert \tilde{g} \Vert_\beta^2 + 8 c^2 \Vert \tilde{Z} \Vert^2_\beta + 8 c^2 \Vert \tilde{M} \Vert^2_{\mathbf{M}_\beta^2}.
\end{equation*}
This inequality joined with the estimates (\ref{estimate1}), give the following estimate for $\vertiii{\tilde{Y}}^2_{\beta}$:
\begin{equation}
\vertiii{\tilde{Y}}^2_{\beta} \leq 2\epsilon^2(1 + 8c^2) \Vert \tilde{g} \Vert^2_{\beta}.
\end{equation}
This completes the proof. 
\end{proof}

\vspace{0.5 cm} From this result, we derive the following uniqueness result for the predictable DRBSDE associated with the driver process $g(t)$.

\begin{theorem}
Let $\xi$ and $\zeta$ be two predictable admissible barriers satisfying Mokobodzki's condition. The predictable DRBSDE (\ref{eq*}) associated with driver process $g(t)$ admits a unique solution $(Y,Z,M,A,B,A',B') \in \mathbf{S}^{2,p} \times \mathbf{H}^2 \times \mathbf{M}^{2,\perp} \times (\mathbf{S}^{2,p})^2 \times (\mathbf{S}^{2,p})^2$. 
\end{theorem}

\begin{proof}
We only need to show the uniqueness of the solution, Theorem \ref{existence} gives the existence. For this purpose, let $(Y,Z,M,A,B,A',B')$ be a solution of the predictable DRBSDE associated with driver process $g(t)$ and obstacles $(\xi, \zeta)$. The previous estimates (\ref{l1}) and (\ref{l2}) in Lemma \ref{lemma uniqueness} (applied with $g=\bar{g}$), gives the uniqueness of $(Y, Z, M)$. The uniqueness of $A, B, A'$ and $B'$ follows from the uniqueness of Mertens decomposition of predictable  
strong supermartingales, which completes the proof. 
\end{proof}


\section{The general case}

\hspace*{0.5 cm} In this section, we are given a Lipschitz driver $g$. We prove existence and uniqueness of the solution to
the predictable DRBSDE from Definition \ref{def1}, in the case of a general Lipschitz driver $g$. The proof is based on a fixed point theorem (applied in an appropriate Banach space) and the estimates given in Lemma \ref{lemma uniqueness}.\\

For each $\beta>0$, we write $\mathbf{K}^2_\beta$ for the space $\mathbf{S}^{2,p} \times \mathbf{H}^2$ equipped with the norm: $$\Vert (Y,Z) \Vert^2_{\mathbf{K}^2_\beta}:= \vertiii{Y}^2_\beta + \Vert Z\Vert_\beta^2, \,\,\,\, \text {for} \,\,\, (Y,Z) \in \mathbf{S}^{2,p} \times \mathbf{H}^2. $$
Note that since $(\mathbf{S}^{2,p},\vertiii{.}^2_\beta )$ and $(\mathbf{H}^2, \Vert .\Vert_\beta^2 )$ are Banach spaces, $\mathbf{K}^2_\beta$ is also a Banach space. 

\begin{theorem}
Let $\xi$ and $\zeta$ be two predictable admissible barriers satisfying Mokobodzki's condition and let $g$ be a Lipschitz driver. There exists a unique solution to the predictable DRBSDE (\ref{eq*}) associated with parameters $(g,\xi,\zeta)$.
\end{theorem}

\begin{proof} 
For $\beta >0$, we introduce a mapping $\Phi$ from $\mathbf{K}^2_\beta$ into itself. This map is defined by: for a given $(U,V)\in \mathbf{K}^2_\beta$, $\Phi(U,V) := (Y,Z)$, where $Y$, $Z$ are the first two components of the solution $(Y,Z,M,A,B,A',B')$ to the predictable DRBSDE associated with driver $g_t:= g(t,U_t,V_t)$ and with the pair of predictable admissible barriers $(\xi,\zeta)$.
 Note that by Theorem \ref{existence}, the mapping $\Phi$ is well-defined. \\
 
 Our goal is to prove that with a convenient choice of the parameter $\beta > 0$, $\Phi$ is a contraction and hence, by the Banach fixed-point theorem, admits a unique fixed point $(Y,Z) \in \mathbf{K}_\beta^2$. By the definition of $\Phi$, the process $(Y,Z)$ will be equal to the first two components of the unique solution $(Y,Z,M,A,B,A',B')$ to the predictable DRBSDE associated with the driver process $h(\omega,t):= g(\omega,t,Y_t(\omega), Z_t(\omega))$ and with the pair of barriers $(\xi,\zeta)$. Thus, we have the existence and uniqueness of the solution to the predictable DRBSDE  (\ref{eq*}). \\
 
  To this end, consider $(U,V)$ and $(\bar{U},\bar{V})$ two elements of $\mathbf{K}^2_\beta$. we set $(Y,Z)= \Phi(U,V)$, $(\bar{Y}, \bar{Z})= \Phi(\bar{U},\bar{V})$, $\tilde{Y}:= Y - \bar{Y}$, $\tilde{Z}:= Z - \bar{Z}$, $\tilde{U}:= U - \bar{U}$ and $\tilde{V}:= V - \bar{V}$. Hence, Lemma \ref{lemma uniqueness} shows that 
 \begin{equation}
 \vertiii{\tilde{Y}}^2_\beta + \Vert \tilde{Z} \Vert^2_{\beta} \leq \epsilon^2(3+ 16 c^2) \Vert g(t, U_t,V_t)- g(t,\bar{U}_t, \bar{V}_t \Vert^2_{\beta}.
 \end{equation}
 By using the Lipschitz property of $g$, we obtain
 \begin{align*}\label{m1}
  \vertiii{\tilde{Y}}^2_\beta + \Vert \tilde{Z} \Vert^2_{\beta}  \leq 2\epsilon^2K^2 (3+ 16 c^2) \Big[\Vert \tilde{U} \Vert^2_{\beta} + \Vert \tilde{V} \Vert^2_{\beta} \Big].
 \numberthis
 \end{align*}
Note that $\Vert \tilde{U} \Vert^2_{\beta} \leq T \vertiii{\tilde{U}}^2_\beta$. Indeed, by Fubini's theorem, we get
\begin{align*}
\Vert \tilde{U} \Vert^2_{\beta} := E[ \int_0^T e^{\beta s} \vert U_s \vert^2 ds ] &=  \int_0^T E[ e^{\beta s} \vert U_s \vert^2] ds; \\ & \leq  \int_0^T E[ \esssup_{\tau \in \mathcal{T}_0^p} e^{\beta \tau} \vert U_\tau \vert^2] ds; \\
& \leq T E[ \esssup_{\tau \in \mathcal{T}_0^p} e^{\beta \tau} \vert U_\tau \vert^2] = T \vertiii{\tilde{U}}^2_\beta.
\end{align*}
This, combined with  (\ref{m1}), gives 
\begin{align*}
  \vertiii{\tilde{Y}}^2_\beta + \Vert \tilde{Z} \Vert^2_{\beta}  \leq 2K^2 (1+T)\epsilon^2(3+ 16 c^2) \Big[ \vertiii{\tilde{U}}^2_\beta + \Vert \tilde{V} \Vert^2_{\beta} \Big].
 \numberthis
 \end{align*}
Consequently, by choosing $\epsilon>0$ such that $2K^2 (1+T)\epsilon^2(3+ 16 c^2)<1$ and $\beta$ such that $\beta \geq \frac{1}{\epsilon^2}$, we deduce that the mapping $\Phi$ is a contraction, which completes the proof.

 
\end{proof}

\section{Appendix}

\hspace*{0.5 cm} Let $T$ be a fixed positive real number. Let $\xi=(\xi_t)_{t\in[0,T]}$ be a predictable process in $\mathbf{S}^{2,p}$, called obstacle or barrier in $\mathbf{S}^{2,p}$. 

\begin{definition}(One barrier predictable reflected BSDE with driver 0) \label{def2} \\ A process $(Y,Z,M,A,B) \in \mathbf{S}^{2,p} \times \mathbf{H}^2 \times \mathbf{M}^{2,\perp} \times (\mathbf{S}^{2,p})^2$ is said to be solution to the predictable reflected BSDE with (lower) barrier $\xi$ and driver 0, if  
\begin{multline} \label{PRBSDE}
Y_\tau=\xi_T - \int_\tau^T Z_s dWs - (M_{T^-}-M_{\tau^-}) +  A_T-A_\tau + B_{T^-}-B_{\tau^-}, 
\end{multline}
with 
\begin{enumerate}
\item[(i)] $ \xi_\tau \leq Y_\tau$ a.s. for all $\tau\in \mathcal{T}_0^p$,
\item[(ii)] $A$ is a nondecreasing right-continuous process with $A_0=0$ and such that 
\begin{equation}
\int_0^T \textbf{1}_{\lbrace Y_t > \xi_t \rbrace} dA^c_t=0 \,\,\, \text{a.s. and     }  (Y_{\tau^-} - \xi_{\tau^-})(A^d_\tau - A^d_{\tau^-})=0 \,\,\, \text{a.s. for all $\tau\in \mathcal{T}_0^p$}, 
\end{equation}
\item[(iii)] $B$ is a nondecreasing right-continuous adapted purely discontinuous process with $B_{0^-}=0$, and such that
\begin{equation}
(Y_\tau-\xi_\tau)(B_\tau -B_{\tau^-})=0 \,\,\, \text{a.s. for all $\tau \in \mathcal{T}_0^p$.} \label{sko2'}
\end{equation}
\end{enumerate}
\end{definition}

The following result established by S. Bouhadou and Y. Ouknine in \cite{Bouhadou} (see Theorem 2, p. 10):

\begin{proposition} \label{existence PBSDE}
Let $\xi$ be a process in $\mathbf{S}^{2,p}$. There exists a unique solution \\ $(Y,Z,M,A,B) \in  \mathbf{S}^{2,p} \times \mathbf{H}^2 \times \mathbf{M}^{2,\perp} \times (\mathbf{S}^{2,p})^2$ of the predictable reflected BSDE from Definition \ref{def2}, and for each stopping time $\tau \in \mathcal{T}^p_0$, we have 
\begin{equation*}
Y_\tau = \esssup_{S \in \mathcal{T}^p_\tau} E(\xi_S \vert \mathcal{F}_{\tau^-}) \,\,\, \text{a.s.}
\end{equation*}
\end{proposition}

 Here are some elementary properties of this operator.
\begin{lemma} \label{incr opera}
The operator $\mathcal{P}re$ is nondecreasing, i.e. for $\xi, \xi' \in \mathbf{S}^{2,p}$ such that $\xi \leq \xi'$ we have $\mathcal{P}re[\xi] \leq \mathcal{P}re[\xi']$. Further, for each $\xi \in \mathbf{S}^{2,p}$, $\mathcal{P}re[\xi]$ is a predictable strong supermartingale and satisfies $\mathcal{P}re[\xi] \geq \xi$.
\end{lemma}

\begin{proof} 
By definition, $\mathcal{P}re[\xi]$ is the first component of the solution of the predictable reflected BSDE (\ref{PRBSDE}). Hence, Theorem 2 in \cite{Bouhadou} shows that $\mathcal{P}re[\xi]$ is the predictable value function associated with the reward $\xi$, that is for each stopping time $S \in \mathcal{T}^p_0$ 
\begin{equation*}
\mathcal{P}re[\xi]_S = \esssup_{\tau \in \mathcal{T}^p_S} E(\xi_\tau \vert \mathcal{F}_{S^-}).
\end{equation*}
Thus, the operator $\mathcal{P}re$ is nondecreasing and the process $(\mathcal{P}re[\xi])_{t\in [0,T]}$  is characterized as the predictable Snell envelope associated with the process $(\xi)_{t\in [0,T]}$, that is the smallest strong predictable supermartingale greater than or equal to $\xi$ (cf.  \cite{Bouhadou} Lemma 15, p. 34) and the lemma follows.
\end{proof}

\begin{remark} \label{pre=}
If $\xi \in \mathcal{T}^p_0$ is a predictable strong supermartingale, then $\mathcal{P}re[\xi]=\xi$. Indeed, it remains to show that
$\mathcal{P}re[\xi] \leq \xi$. Let $S \in \mathcal{T}^p_0 $, since $\xi $ is a predictable strong supermartingale, for each stopping time $\tau \in \mathcal{T}^p_S$, we have
$$ E(\xi_\tau \vert \mathcal{F}_{S^-}) \leq \xi_S.$$
By definition of the essential supremum, we get $\mathcal{P}re[\xi]_S \leq \xi_S$. Consequently, $\mathcal{P}re[\xi] = \xi$.
\end{remark}

\begin{remark} \label{cv seq of p.s.s}
The limit of a nondecreasing sequence of predictable strong supermartingales is also a predictable strong supermartingale (It can be shown using Lebesgue's dominated convergence theorem and the fact that every trajectory of a predictable strong supermartingale is bounded on all compact interval of $\mathbb{R}^+$). 
\end{remark}

\begin{proofa} Let $n\in \mathbb{N}$, we begin by proving that the processes $\mathbf{J}^{p,n}$ and $\bar{\mathbf{J}}^{p,n}$ are valued in $[0,+\infty]$. By definition, we have: 
\begin{equation}\label{T}
\mathbf{J}^{p,n}_T=\bar{\mathbf{J}}^{p,n}_T=0 \,\,\,\, \text{a.s. for each} \,\,\,\, n.
\end{equation} 
Hence, $\mathbf{J}^{p,n}$ and $\bar{\mathbf{J}}^{p,n}$ are non negative since they are predictable strong supermartingales. From $\tilde{\xi}^{p,g}_T=\tilde{\zeta}^{p,g}_T=0$, it follows that $(\bar{\mathbf{J}}^{p,n} + \tilde{\xi}^{g,p} ) \textbf{1}_{[0,T)}=(\bar{\mathbf{J}}^{p,n} + \tilde{\xi}^{g,p} )$ and $(\mathbf{J}^{p,n}  -\tilde{\zeta}^{g,p} ) \textbf{1}_{[0,T)}=(\mathbf{J}^{p,n}  -\tilde{\zeta}^{g,p} )$. 
We prove that $(\mathbf{J}^{p,n})_{n\in \mathbb{N}}$ and $(\bar{\mathbf{J}}^{p,n})_{n\in \mathbb{N}}$ are non decreasing sequences of processes. \\

We have $\mathbf{J}^{p,0}=0 \leq \mathbf{J}^{p,1}$ and $\bar{\mathbf{J}}^{p,0}=0 \leq \bar{\mathbf{J}}^{p,1}$. Suppose that $\mathbf{J}^{p,n-1} \leq \mathbf{J}^{p,n}$ and $\bar{\mathbf{J}}^{p,n-1} \leq \bar{\mathbf{J}}^{p,n}$. The nondecreasingness of the operator $\mathcal{P}re$ gives 
\begin{align*}
&\mathcal{P}re[\mathbf{J}^{p,n-1}  -\tilde{\zeta}^{g,p} ] \leq \mathcal{P}re[\mathbf{J}^{p,n}  -\tilde{\zeta}^{g,p} ] \\ &\mathcal{P}re[\bar{\mathbf{J}}^{p,n-1} + \tilde{\xi}^{g,p}] \leq \mathcal{P}re[\bar{\mathbf{J}}^{p,n} + \tilde{\xi}^{g,p}].
\end{align*}
Thus, $\mathbf{J}^{p,n-1} \leq \mathbf{J}^{p,n}$ and $\bar{\mathbf{J}}^{p,n-1} \leq \bar{\mathbf{J}}^{p,n}$, which is the desired conclusion. \\

The processes $\mathbf{J}^p := \lim \uparrow \mathbf{J}^{p,n}  $ and $ \bar{\mathbf{J}}^p := \lim \uparrow \bar{\mathbf{J}}^{p,n}$ are predictable (valued in $[0,+\infty]$) as the limit of sequences of predictable nonnegative processes. By (\ref{T}), we get $\mathbf{J}^p _T = \bar{\mathbf{J}}^p_T=0 $ a.s. Moreover, $\mathbf{J}^p $ and $\bar{\mathbf{J}}^p$ are strong supermartingales valued in $[0, +\infty]$ (cf. Remark \ref{cv seq of p.s.s}).\\
\hspace*{0.5 cm} We next prove that $\mathbf{J}^p $ and $ \bar{\mathbf{J}}^p$ belong to $\mathbf{S}^{2,p}$. For this purpose, consider $H^{p}$ and $\bar{H}^{p}$ the nonnegative predictable strong supermartingales that come from Mokobodzki's condition for $(\xi,\zeta)$. Then, we define two processes $H^{g,p}$ and $\bar{H}^{g,p}$ as follows: 
\begin{align*}
H^{g,p}_t &:= H^p_t + E[\xi^-_T \vert \mathcal{F}_{t^-}] + E[ \int_t^T g^-(s) ds \vert \mathcal{F}_{t^-}] \\
\bar{H}^{g,p}_t &:= \bar{H}^p_t + E[\xi^+_T \vert \mathcal{F}_{t^-}] + E[ \int_t^T g^+(s) ds, \vert \mathcal{F}_{t^-}].
\end{align*}
It is easy to check that $H^{g,p}$ and $\bar{H}^{g,p}$ are nonnegative predictable strong supermartingales in $\mathbf{S}^{2,p}$. From Mokobodzki's condition, we get
\begin{equation} \label{M}
\tilde{\xi}^{g,p} \leq H^{g,p} - \bar{H}^{g,p} \leq\tilde{\zeta}^{g,p}.
\end{equation}
Let us now show by induction that $\mathbf{J}^{p,n} \leq H^{g,p}$ and $\bar{\mathbf{J}}^{p,n} \leq \bar{H}^{g,p}$, for all $n \in \mathbb{N}$. First, we have $\mathbf{J}^{p,0}=0 \leq H^{g,p}$ and $\bar{\mathbf{J}}^{p,0}=0 \leq \bar{H}^{g,p}$. Suppose that, for a fixed $n \in \mathbb{N}$, we have $\mathbf{J}^{p,n} \leq H^{g,p}$ and $\bar{\mathbf{J}}^{p,n} \leq \bar{H}^{g,p}$. From equation (\ref{M}), we get $\mathbf{J}^{p,n} \leq \bar{H}^{g,p} + \tilde{\zeta}^{g,p} $ and $ \bar{\mathbf{J}}^{p,n} \leq  H^{g,p} - \tilde{\xi}^{g,p}$. As the operator $\mathcal{P}re$ is a non decreasing operator (see Lemma \ref{incr opera}), we get   
\begin{align*}
 \mathbf{J}^{p,n+1} = \mathbf{P}re[\bar{\mathbf{J}}^{p,n} + \tilde{\xi}^{g,p}] \leq \mathcal{P}re[H^{g,p}] \,\,\,\, \text{and}  \,\,\,\,  \bar{\mathbf{J}}^{p,n+1} = \mathbf{P}re[\mathbf{J}^{p,n} - \tilde{\zeta}^{g,p}] \leq \mathcal{P}re[\bar{H}^{g,p}].
\end{align*}\normalsize
Since $H^{g,p} $ and $ \bar{H}^{g,p}$ are predictable strong supermartingales, it follows by Remark \ref{pre=} that $\mathcal{P}re[H^{g,p}]= H^{g,p}$ and $\mathcal{P}re[\bar{H}^{g,p}]= \bar{H}^{g,p}$. Hence, $\mathbf{J}^{p,n+1} \leq H^{g,p}$ and $\bar{\mathbf{J}}^{p,n+1} \leq \bar{H}^{g,p}$, which is the desired conclusion.\\
By letting $n$ tend to $+ \infty$ in  $\mathbf{J}^{p,n} \leq H^{g,p}$ and $\bar{\mathbf{J}}^{p,n} \leq \bar{H}^{g,p}$, we get  $\mathbf{J}^{p} \leq H^{g,p}$ and $\bar{\mathbf{J}}^{p} \leq \bar{H}^{g,p}$. Hence, $\mathbf{J}^{p}$ and $\bar{\mathbf{J}}^{p}$ belong to $\mathbf{S}^{2,p}$. \\

The proof is completed by showing that the processes $\mathbf{J}^{p}$ and $\bar{\mathbf{J}}^{p}$ satisfy the system (\ref{coupled*}). Note that $(\bar{\mathbf{J}}^{p,n} + \tilde{\xi}^{g,p})_{n\in \mathbb{N}}$ is a non decreasing sequence of processes belonging to $\mathbf{S}^{2,p}$. As the operator $\mathcal{P}re$ is a nondecreasing, the sequence $(\mathcal{P}re[\bar{\mathbf{J}}^{p,n} + \tilde{\xi}^{g,p}])_{n\in \mathbb{N}}$ is also nondecreasing. Hence, for each $n \in \mathbb{N}$, the following property 
$$\mathbf{J}^{p,n+1} =\mathcal{P}re[\bar{\mathbf{J}}^{p,n} + \tilde{\xi}^{g,p}] \leq \mathcal{P}re[\bar{\mathbf{J}}^p + \tilde{\xi}^{g,p}], $$
holds. By letting $n$ go to $+\infty$, we get 
\begin{equation}\label{<}
\mathbf{J}^{p} \leq \mathcal{P}re[\bar{\mathbf{J}}^p + \tilde{\xi}^{g,p}].
\end{equation}
 Now, by definition of $\mathcal{P}re[\bar{\mathbf{J}}^{p,n} + \tilde{\xi}^{g,p}]$ as the solution of the predictable reflected BSDE with obstacle $\bar{\mathbf{J}}^{p,n} + \tilde{\xi}^{g,p}$, we have $\mathcal{P}re[\bar{\mathbf{J}}^{p,n} + \tilde{\xi}^{g,p}] \geq \bar{\mathbf{J}}^{p,n} + \tilde{\xi}^{g,p}$, for all $n\in\mathbb{N}$. Thus, by letting $n$ go to $+\infty$, we get $\mathbf{J}^{p} \geq \bar{\mathbf{J}}^p + \tilde{\xi}^{g,p}$. Hence, 
\begin{equation}
\mathcal{P}re[\mathbf{J}^{p}] \geq \mathcal{P}re[\bar{\mathbf{J}}^p + \tilde{\xi}^{g,p}]
\end{equation}
Since $\mathbf{J}^{p}$ is a predictable strong supermartingale, Remark \ref{pre=} implies $\mathcal{P}re[\mathbf{J}^{p}]= \mathbf{J}^{p}$. From inequality (\ref{<}), It follows that $\mathbf{J}^{p}= \mathcal{P}re[\bar{\mathbf{J}}^p + \tilde{\xi}^{g,p}]$.
 We show similarly that  $\bar{\mathbf{J}}^{p}= \mathcal{P}re[\mathbf{J}^p + \tilde{\zeta}^{g,p}]$. Since, $\mathbf{J}^{p}_T= \bar{\mathbf{J}}^{p}_T=0$, we conclude that $\mathbf{J}^{p}$ and $\bar{\mathbf{J}}^{p}$ are solutions of the system  (\ref{coupled*}), and the lemma follows.
\end{proofa}

\begin{acknowledgements}
The authors thank the anonymous Referee for his valuable comments and suggestions from which the manuscript greatly benefited.
\end{acknowledgements}

%
%



%
%




\end{document}